\documentclass[12pt,a4paper]{article}
\usepackage[english]{babel}
\usepackage[latin1]{inputenc}
\usepackage[all]{xy}
\usepackage{amsmath}
\usepackage{amssymb}
\usepackage{amsthm}
\usepackage{graphicx}
\usepackage{mathrsfs}
\usepackage{amsmath,amsfonts,amssymb,amsthm}
\newtheorem{thm}{Theorem}
\newtheorem{de}{Definition}
\newtheorem{rem}{Remark}

\newtheorem{conj}{Conjecture}
\newtheorem{lem}{Lemma}
\newtheorem{prop}{Proposition}
\newtheorem{hyp}{Hypothesis}
\newtheorem{cor}{Corollary}
\newcommand{\ov}{\overline}

\title{Lehmer problem and Drinfeld modules}
\author{Luca Demangos\\Department of Mathematical Sciences (Mathematics Division),\\
University of Stellenbosch,
Private Bag X1,
Matieland 7602,
South Africa\\l.demangos@gmail.com}

\begin{document}
\maketitle
\begin{abstract}
We propose a lower bound estimate in Dobrowolski's form of the canonical height of a Drinfeld module having a positive density of supersingular primes. This estimate takes into account the inseparable case and it is given as a function of: the degree of the field of coefficients, the height of the module and its rank. We will show that the class of Drinfeld modules we consider includes 
all CM Drinfeld modules with rank either 1 or a prime number different from the field characteristic.
\end{abstract}
\section*{Acknowledgement}
This work was made during author's permanence as a PhD student at the Laboratoire Paul Painlevé - Université des Sciences et Technologies de Lille, France. The author gratefully thanks Laurent DENIS (Laboratoire Paul Painlevé), Vincent BOSSER (LMNO, Caen), Sinnou DAVID (Institut de Mathématiques Université Pierre et Marie Curie) and Hugues BAUCHERE (LMNO, Caen) for the fundamental help and support. This research was funded by GTEM project (Marie Curie actions).
\section{Introduction}
We study the natural analogue of Lehmer problem on Drinfeld modules. We consider in particular a special class of such modules, satisfying congruence properties that, for a suitable positive real number $r$, we call RV($r$) or RV($r$)$^{*}$ (see Definition 4 and Definition 5 below). We will call $A:=\mathbb{F}_{q}[T]$ the polynomial ring in one variable $T$ defined over the finite field of $q$ elements, where $q$ is a power of a chosen prime number $p$. We also call $k$ the fraction field of $A$, and $k_{\infty}=\mathbb{F}_{q}((1/T))$ the completion of $k$ with respect to the place at infinity. Let us call $\mathcal{C}:=(\ov{k_{\infty}})_{\infty}$ the completion of a chosen algebraic closure of $k_{\infty}$. This field is therefore algebraically closed and complete.\\\\
The main result we propose in this work (Theorem 2) is the following. Given a Drinfeld module (see Definition 1) $\mathbb{D}=(\mathbb{G}_{a},\Phi)$ satisfying suitable congruence properties involving the density of supersingular primes (see Definition 4 and Definition 5), we provide a lower bound estimate of the canonical height of a non-torsion point $x\in \mathbb{D}(\ov{k})$ with respectively algebraic degree and purely inseparable degree $D$ and $D_{p.i.}$ over $k$ in the following form:\[C\frac{(\log\log{D})^{\mu}}{DD_{p.i.}^{\lambda}(\log{D})^{\kappa}},\]where the positive constants $C$, $\kappa$, $\mu$ and $\lambda$ are explicitly computed as functions of the three arithmetic parameters attached to $\mathbb{D}$: the degree of the field of coefficients $k(\Phi)$, the height $h(\Phi)$ of the Drinfeld module, and the rank $d$.\\\\
Let:\[\tau:\mathcal{C}\to \mathcal{C}\]\[z\mapsto z^{q}\]be the Frobenius map and:\[\ov{k}\{\tau\}:=\{c_{0}+c_{1}\tau + ... + c_{n}\tau^{n},\texttt{   }c_{1}, ..., c_{n}\in \ov{k},\texttt{   }n\in \mathbb{N}\}\]be the \textsl{Ore algebra} of the $\mathbb{F}_{q}-$additive forms with coefficients in $\ov{k}$\footnote{We remark that such an algebra is not commutative as of course in general for $c\in \ov{k}$ one has $\tau c=c^{q}\tau\neq c\tau$.}.
\begin{de}
A \textbf{Drinfeld module} of \textbf{rank $d$} defined over $\ov{k}$ is a pair:\[\mathbb{D}=(\mathbb{G}_{a},\Phi),\]where $\mathbb{G}_{a}$ is the additive group of $\mathcal{C}$ and $\Phi$ is an injective $\mathbb{F}_{q}-$algebra homomorphism:\[\Phi:A\to \ov{k}\{\tau\},\]defined so that:\[\Phi(T)=\sum_{i=0}^{d}a_{i}\tau^{i}\]where $a_{0}, ..., a_{d}\in \ov{k}$ are such that:\[a_{0}(T)=T\texttt{ and }a_{d}(T)\neq 0.\]We call $k(\Phi):=k(a_{1}, ..., a_{d})$ the \textbf{field of coefficients} of $\mathbb{D}$ (alternatively, we say that $\mathbb{D}$ is defined over $k(\Phi)$).
\end{de}
We call \textbf{torsion point} of the Drinfeld module $\mathbb{D}=(\mathbb{G}_{a},\Phi)$ a point $x\in \ov{k}$ such that there exists $a\in A\setminus\{0\}$ for which we have:\[\Phi(a)(x)=0.\]In particular, we say that $x$ is a \textbf{$a-$torsion point} for this $a$. We also denote by $\Phi[a]$ the set (which is in particular a $\mathbb{F}_{q}-$vector space) of the $a-$torsion points of $\mathbb{D}$. We define:\[\mathbb{D}(\ov{k})_{NT}:=\ov{k}\setminus \bigcup_{a\in A\setminus\{0\}}\Phi[a]\]the set of non-torsion points of $\mathbb{D}$.\\\\
The \textbf{Carlitz module} $C=(\mathbb{G}_{a},\Phi)$ is defined so that:\[\Phi(T)=T+\tau\]and it is the simplest example of a Drinfeld module having rank $1$.\\\\
The \textbf{Lehmer conjecture} in its original form concerns the multiplicative group $\mathbb{G}_{m}(\ov{\mathbb{Q}})$ and predicts a bound taking this shape:\[h(x)>>\frac{1}{[\mathbb{Q}(x):\mathbb{Q}]}\]for all $x\in \mathbb{G}_{m}(\ov{\mathbb{Q}})$ which are not roots of unity.\\\\
Different versions of this conjecture have been proposed, in particular stating lower bounds of the same shape for the Néron-Tate height of non-torsion points of an abelian variety. L. Denis conjectured in \cite{Denis} the following analogue for the canonical height (see the second paragraph for the definition) of the algebraic non-torsion points of a general Drinfeld module defined over $\ov{k}$:
\begin{conj}
There exists a constant $c>0$ only depending on the Drinfeld module $\mathbb{D}=(\mathbb{G}_{a},\Phi)$, such that each point $x\in\mathbb{D}(\ov{k})_{NT}$ of degree $D$ over $k$ satisfies the following inequality:\[\widehat{h}_{\mathbb{D}}(x)\geq\frac{c}{D}.\]
\end{conj}
For the specific case of the Carlitz module L. Denis also obtained in the same paper the following:
\begin{thm}[Denis]
Let $\mathbb{D}$ be the Carlitz module. There exists $\eta>0$ depending on $q$ such that for each $x$ non-torsion algebraic and separable point with degree $\leq D$ over $k$:\[\widehat{h}_{\mathbb{D}}(x)\geq\frac{\eta}{D}(\frac{\log\log(qD)}{\log(qD)})^{3}.\]
\end{thm}
D. Ghioca (see \cite{Gh}, Remark 5) showed moreover, with no conditions on the Drinfeld module but on a strong local condition on $x$, that there exists a number $k\geq 1$, depending only on the chosen Drinfeld module, such that:\[\widehat{h}_{\mathbb{D}}(x)>>\frac{1}{D^{k}}.\]\\\\
Another result has also been found recently by S. David and A. Pacheco (see \cite{Dav-Pach}) who showed the following lower bound estimate:\[\widehat{h}_{\mathbb{D}}(x)\geq c(\mathbb{D},K)\]for a Drinfeld module $\mathbb{D}$ defined over the field $K\subset \ov{k}$, where $c(\mathbb{D},K)>0$ is a positive constant only depending on $\mathbb{D}$ and $K$, and $x\in K^{ab.}$, where $x$ is non-torsion and $K^{ab.}$ is the abelian closure of $K$ in $\ov{k}$. Such a result is in analogy with the work of F. Amoroso and R. Dvornicich (see \cite{Am-Dv}) which provides an estimate of this form for the height of an element $x\in \mathbb{G}_{m}(\mathbb{Q}^{ab.})\setminus \mathbb{G}_{m}(\mathbb{Q}^{ab.})_{tors.}$.\\\\
We give now the fundamental notations about the logarithmic functions we will use:\[\log(.):=\log_{q}(.).\]Each logarithm will have always basis $q$ unless we specify differently.\[\log_{+}(.):=\max\{\log(.),1\}\]\[\log\log_{+}(.):=\max\{\log\log(.),1\}.\]We will indicate from now on the degree in $T$ of each polynomial $a\in A=\mathbb{F}_{q}[T]$ by $\deg_{T}(a)$.\\\\
We define:\[S(A):=\{l\in A, \texttt{ monic and irreducible}\}.\]We also define, given some $N\in \mathbb{N}\setminus\{0\}$
:\[P_{N}(A):=\{l\in S(A), \deg_{T}(l) =N\}.\]We will also say that $l\in$ $S(A)$ \textbf{satisfies the RV property}\footnote{The acronym "RV" has been suggested to the author by the french word \textsl{relèvement} he was using to describe the property 1 of the primes involved, which means that there is a "lifting" of the $d-$th power of the Frobenius automorphism by an endomorphism of $\Phi$.} with respect to $\Phi$ if:\begin{enumerate}
	\item For each place $v$ dividing $v_{l}$ (the place associated to $l$ over $k$) in the extension $k(\Phi)/k$, the coefficients $a_{i}$ of $\Phi$ are such that $v(a_{i})\geq 0$ and:\[\Phi(l)(X)\equiv X^{q^{d\deg_{T}(l)}} \texttt{ mod }(v)\]where:\[\Phi(l)(X)\in \mathcal{O}_{v}[X],\]
the ring $\mathcal{O}_{v}$ being the ring of $v-$integers in $k(\Phi)$; 
	\item All places extending $v_{l}$ in $k(\Phi)$ have inertial degree $1$. 
\end{enumerate}
\begin{de}
Let $r\in ]0,1]$ be a real number and $c_{1}$ a fixed positive constant. A Drinfeld module $\mathbb{D}=(\mathbb{G}_{a},\Phi)$ is called \textbf{RV($r,c_{1}$)}, 
if for each natural number $N>0$:\[|\{l\in P_{N}(A), l\texttt{ is }RV\}|\geq c_{1}\frac{q^{rN}}{N}.\]
\end{de}
\begin{de}
A Drinfeld module $\mathbb{D}=(\mathbb{G}_{a},\Phi)$ is \textbf{RV($r,c_{1}$)$^{*}$}, with $r\in ]0,1]$ and $c_{1}>0$ a fixed constant, if there exists $N(\Phi)\in \mathbb{N}\setminus\{0\}$ such that, for each $N\geq N(\Phi)$:\[|\{l\in P_{N}(A), l\texttt{ is }RV\}|\geq c_{1}\frac{q^{rN}}{N}.\]
\end{de}
We fix $c_{1}=1/2r$ to ease notations reducing the number of parameters. We leave to the reader a generalization (not quite relevant) of the proposed estimates to a more general $c_{1}$. This directly follows by a mechanical repetition of the same steps of our argument. The choice of $c_{1}=1/2r$ has been suggested by the fact that, as we will see soon (Proposition 3) the value $c_{1}=1/2$ is the maximal that one can choose if $r=1$.
\begin{de}
Let $r\in ]0,1]$ be a real number. A Drinfeld module $\mathbb{D}=(\mathbb{G}_{a},\Phi)$ is \textbf{RV($r$)} if it is RV($r,1/2r$).
\end{de}
\begin{de}
Let $r\in]0,1]$ be a real number. A Drinfeld module $\mathbb{D}=(\mathbb{G}_{a},\Phi)$ is \textbf{RV($r$)$^{*}$} if it is RV($r,1/2r$)$^{*}$. 
\end{de}
It is clear that the condition RV($r$)$^{*}$ is implied by the RV($r$) one for every $r\in ]0,1]$. We also remark that the Carlitz module is RV($1$). Indeed, one can prove (see \cite{Hayes}, Proposition 2.4) that each $l\in S(A)$ has supersingular reduction with respect to the Carlitz module. In particular, the Carlitz module will satisfy our Theorems.\\\\
A result extending the study to rank 2 Drinfeld modules was showed by C. David in \cite{C. David}: in average, a rank 2 Drinfeld module with coefficients in $k$ (note that under this condition the properties for a prime to be supersingular or RV are equivalent) satisfies the analogue of the Lang-Trotter conjecture (in simple terms, the growth of the number of supersingular reduction primes takes the shape:\[|\{l\in P_{N}(A),\texttt{ }l\textsl{ is supersingular}\}|\sim_{N\to +\infty}\frac{q^{N/2}}{N}).\]
This provides a considerable number of examples, in rank $2$, satisfying the RV($r,c_{q}$)$^{*}$ condition, with $r=1/d=1/2$ for some constant $c_{q}>0$ only depending on $q$. We point out anyway that this conjecture for Drinfeld modules is \textbf{false} (yet remaining open in the "classic" case of the elliptic curves), for each possible value of the rank, as a consequence of the remarkable work of B. Poonen, \cite{B. Poonen}.\\\\
The methods that we will present in the Appendix will also show that the class of Drinfeld modules with complex multiplication having either rank $1$ or a prime number different from the characteristic of $k$ is contained in RV($1,1/2d$)$^{*}$.\\\\
We will use the following notation for the degree of the extension fields which will be involved:\[D=[k(x):k],\texttt{   }c(\Phi):=[k(\Phi):k],\texttt{   }D':=[k(\Phi)(x):k(\Phi)].\]We also call:\[D_{p.i.}:=[k(x):k]_{p.i.}\]the inseparable degree of $x$ over $k$,\[D'_{p.i.}:=[k(\Phi)(x):k(\Phi)]_{p.i.}\]the inseparable degree of $x$ over $k(\Phi)$ and:\[D'_{sep.}:=[k(\Phi)(x):k(\Phi)]_{sep.}\]the separable degree of $x$ over $k(\Phi)$. We have that:\[D'=D'_{sep.}D'_{p.i.}.\]We also call $h(\Phi)$ the height of our Drinfeld module (see Section 2 for the definition). 
We now state our main result in this work.
\begin{thm}
Let $\mathbb{D}=(\mathbb{G}_{a},\Phi)$ be a Drinfeld module defined over $\ov{k}$ satisfying the hypothesis $RV(r)$ or $RV(r)^{*}$. Let:\[c_{0}:=35000dh(\Phi)^{3}c(\Phi)^{3}q^{d+rh(\Phi)c(\Phi)}\]and:\[C_{0}:=\min\{q^{-5d(2(d+1)h(\Phi)+1)((q^{q+d+1}-1)c(\Phi))^{2}},\frac{h(\Phi)}{384rq^{d}c_{0}^{\frac{4h(\Phi)c(\Phi) d}{r}+1}}\}.\]Then, there exists $C>0$ such that for all $x\in\mathbb{D}(\ov{k})_{NT}$ one has:\[\widehat{h}_{\mathbb{D}}(x)\geq C\frac{(\log\log_{+}{D})^{\mu}}{D{D_{p.i.}}^{\lambda}(\log_{+}{D})^{\kappa}}\]where:\begin{equation}\mu:=2+\frac{d}{r}h(\Phi)c(\Phi);\label{eq:5}\end{equation} \begin{equation}\kappa:=1+\frac{3d}{r}h(\Phi)c(\Phi);\label{eq:6}\end{equation}\begin{equation}\lambda:=1+\frac{2d}{r}h(\Phi)c(\Phi);\label{eq:7}\end{equation}and:\[C=C_{0}\texttt{   under the hypothesis   }RV(r)\]while\[0<C\leq C_{0}\texttt{   under the hypothesis   }RV(r)^{*}.\]
\end{thm}
As $D_{p.i.}\leq D$ we conclude that:
\begin{cor}
Under the same hypotheses of Theorem 2 we have:\[\widehat{h}(x)\geq C\frac{(\log\log_{+}{D})^{\mu}}{D^{1+\lambda}(\log_{+}{D})^{\kappa}}.\]
\end{cor}
We note (see for example \cite{P1}, Proposition 2) that any lower bound in Dobrowolski's form of the canonical height associated to a Drinfeld module (and in particular our bound as well) extends essentially to the whole isogeny class of such a module, up to slight modifications of the multiplicative constant, depending on the degree of the isogeny. In particular, there are no changes at all to the multiplicative constant between isomorphic Drinfeld modules (case in which the isogeny degree is $0$, see \cite{Goss}, Chapter 4). 

\section{Preliminary results}
Let $\mathbb{P}^{1}(\ov{k})$ be the projective line defined over $\ov{k}$. If we take a \textbf{place} $v$ over $k$, it is well known that it is associated to an irreducible element $l\in A\setminus \{0\}$ or to the point $\infty\in \mathbb{P}^{1}(k)$ so that in the first case we have:\[v(x):=\deg_{T}(l)v_{l}(x)\texttt{     }\forall x\in k,\]where $v_{l}(x)$ is the $l-$divisibility index of $x$; while in the other one:\[v(x):=v_{\infty}(x):=-\deg_{T}(x)\texttt{     }\forall x\in k.\]Each one of such places has finitely many extensions to a finite field extension $L$ of $k$. 
Now, for each $x\in \ov{k}$ and each place $w$ over $k(x)/k$ restricting to $v$ in $k$, one defines:\[n_{w}:=[k(x)_{w}:k_{v}],\]where $k_{v}$ and $k(x)_{w}$ are respectively the completion of $k$ with respect to $v$ and the completion of $k(x)$ with respect to $w$. We recall the well-known facts that:\[[k(x):k]=\sum_{w|v}n_{w}\texttt{ and }n_{w}=e_{w}f_{w},\]where $e_{w}$ and $f_{w}$ are respectively the ramification index and the inertial degree of $w|v$. We note that:\[v(k^{*})\subseteq \mathbb{Z}\texttt{ and }w(k(x)^{*})\subseteq \frac{1}{e_{w}}\mathbb{Z}.\]Moreover, for every $\alpha\in k$ and every $w|v$, we have that $v(\alpha)=w(\alpha)$. 
The \textbf{height} of $x$ is defined as follows:\[h(x)=\frac{1}{D}\sum_{w\texttt{ over }k(x)/k}n_{w}\max\{0,-w(x)\},\]where $D=[k(x):k]$. By writing "$w$ over $k(x)/k$" we mean that the sum is on all the places extending in $k(x)$ every place over $k$ as described above. The more general definition of height of $\ov{x}=(x_{1}, ..., x_{n})\in \ov{k}^{n}$ (for some $n>1$) is the following:\[h(\ov{x}):=\frac{1}{D}\sum_{w\texttt{ over } k(\ov{x})/k}n_{w}\max_{i=1, ..., n}\{0,-w(x_{i})\},\]where $D=[k(\ov{x}):k]$ and $k(\ov{x})=k(x_{1}, ..., x_{n})$. 
We list the main properties of the logarithmic height over $\ov{k}^{n}$ which will be needed in our proof. We start by introducing the following notation we will use along the entire text. Let $\ov{a}=(a_{1}, ..., a_{n})$ and $\ov{b}=(b_{1}, ..., b_{n})$ be two vectors of $n$ components, for any fixed positive integer $n$. We introduce the following notation:\[\ov{a}*\ov{b}:=(a_{1}b_{1}, ..., a_{n}b_{n}).\]
\begin{prop}
\begin{enumerate}
  \item Let $\ov{\alpha}, \ov{\beta}\in \ov{k}^{n}$. We have that:\[h(\ov{\alpha}+\ov{\beta})\leq h(\ov{\alpha})+h(\ov{\beta}).\]
	\item Let $\ov{\alpha},\ov{\beta}\in \ov{k}^{n}$. Then we have:\begin{equation} h(\ov{\alpha}*\ov{\beta})\leq h(\ov{\alpha})+h(\ov{\beta}). \label{eq:2} \end{equation} 
	\item  Let $\ov{\alpha},\ov{\beta}\in \ov{k}^{n}$. Let $(\ov{\alpha},\ov{\beta})\in \ov{k}^{2n}$ be the vector of $2n$ entries obtained by "glueing" $\ov{\alpha}$ with $\ov{\beta}$. Then: \begin{equation} h(\ov{\alpha}+\ov{\beta})\leq h(\ov{\alpha},\ov{\beta}). \label{eq:3} \end{equation} 
\end{enumerate}
\end{prop}
These properties are easily implied by the previous definitions.\\\\
The \textbf{height $h(\Phi)$ of a Drinfeld module} $\mathbb{D}=(\mathbb{G}_{a},\Phi)$, where $\Phi(T)$ has coefficients $T, a_{1}, ..., a_{d}\in k(\Phi)$, is:\[h(\Phi)=h(T, a_{1},...,a_{d}).\]One can easily see that $h(\Phi)\geq 1$. 
The \textbf{Néron-Tate height}, or \textbf{canonical height} of a Drinfeld module $\mathbb{D}=(\mathbb{G}_{a},\Phi)$ with rank $d$ has been introduced by 
L. Denis \cite{Denis} as follows:\[\widehat{h}_{\mathbb{D}}(x)=\lim_{n\to\infty}\frac{h(\Phi(T^{n})(x))}{q^{dn}}.\]We replace from now on the notation "$\widehat{h}_{\mathbb{D}}$" by simply "$\widehat{h}$" as 
in the entire text there will be no reference to other possible Drinfeld modules. 
\begin{prop}
Let $\mathbb{D}=(\mathbb{G}_{a},\Phi)$ be a Drinfeld module of rank $d$, such that:\[\Phi(T)(\tau)=T+a_{1}(T)\tau+...+a_{d}(T)\tau^{d}.\]We set:\[\gamma(\Phi):=\sup_{x\in \ov{k}}|h(x)-\widehat{h}(x)|.\]Then:\[\gamma(\Phi)<2(d+1)h(\Phi).\]
\end{prop}
\begin{proof}
See \cite{D}, Théorème 1.2.7. 
\end{proof}
We give now a first rough lower bound estimate of the canonical height of a Drinfeld module. 
\begin{lem}
For each $\chi \geq 1$, $D\geq 1$, we have:\[|\{x\in\ov{k}, [k(x):k]\leq D,h(x)\leq\chi\}|\leq q^{5D^{2}\chi}.\]
\end{lem}
\begin{proof}
See \cite{D}, Lemme 1.2.9.
\end{proof}

\begin{lem}
Let $\mathbb{D}=(\mathbb{G}_{a},\Phi)$ be a Drinfeld module with rank $d$. By taking $c_{2}=q^{5d(2(d+1)h(\Phi)+1)c(\Phi)^{2}}$ we have, for all $x\in \mathbb{D}(\ov{k})_{NT}$ of degree $D$ over $k$:\[\widehat{h}(x)\geq\frac{1}{c_{2}^{D^{2}}}.\]
\end{lem}
\begin{proof}
We proceed by contradiction. Assume that $\widehat{h}(x)<\frac{1}{c_{2}^{D^{2}}}=\frac{1}{q^{c_{3}D^{2}c(\Phi)^{2}}}$. Therefore for any $a\in A$ we have that:\[\widehat{h}(\Phi(a)(x))=q^{d\deg_{T}(a)}\widehat{h}(x)<\frac{q^{d\deg_{T}(a)}}{q^{c_{3}D^{2}c(\Phi)^{2}}}.\]Let us choose:\[\deg_{T}(a)\leq \frac{c_{3}D^{2}c(\Phi)^{2}}{d}.\]Therefore, $\widehat{h}(\Phi(a)(x))<1$ and thus, by Proposition 2:\[h(\Phi(a)(x))\leq 1+\gamma(\Phi)\leq 1+2(d+1)h(\Phi).\]Lemma 1 allows us to say that the number of elements $y$ algebraic with degree $\leq Dc(\Phi)$ over $k$ such that $h(y)\leq 1+2(d+1)h(\Phi)$, is at most $q^{5(1+2(d+1)h(\Phi))D^{2}c(\Phi)^{2}}$. Now all elements of the form $y=\Phi(a)(x)$ are algebraic of degree $[k(\Phi)(x):k]\leq Dc(\Phi)$. Moreover, since $x$ is non-torsion (so that if $a\neq b$ then $\Phi(a)(x)\neq\Phi(b)(x)$), we have for all positive integers $M$:\[\left|\bigcup_{a\in A,\deg_{T}(a)\leq M}\{\Phi(a)(x)\}\right|=q^{M+1}.\]Choosing $M=[\frac{1}{d}(c_{3}D^{2}c(\Phi)^{2})]$ we obtain $q^{M+1}$ distinct elements with degree over $k$ at most $Dc(\Phi)$ and height at most $1+2(d+1)h(\Phi)$. We also know that such a set contains at most $q^{5(1+2(d+1)h(\Phi))D^{2}c(\Phi)^{2}}$ elements. Thus we obtain 
:\[[\frac{c_{3}D^{2}c(\Phi)^{2}}{d}]+1\leq 5(1+2(d+1)h(\Phi))D^{2}c(\Phi)^{2}\]which yields a contradiction and proves the statement by our choice of $c_{3}$. 
\end{proof}
\begin{prop}
Let $X$ be the number of monic, irreducible polynomials in $A$ with degree $N$, for $N\in\mathbb{N}\setminus\{0\}$. Then:\[\frac{1}{2}\frac{q^{N}}{N}\leq X\leq \frac{q^{N}}{N}.\]
\end{prop}
\begin{proof}
The exact value of $X$ as a function of $N$ is:\[X=1/N\sum_{d|N}\mu(N/d)q^{d}\]where $\mu$ is the Moebius function, see \cite{Irrose}, page 84. Therefore, for each $d|N$, $\mu(N/d)\leq1$ where\footnote{If $N=1$ we remark that $X=q$, which satisfies our statement.} $N\geq 2$:\[\left|1/N\sum_{d|N,d\neq N}\mu(\frac{N}{d})q^{d}\right|\leq\frac{1}{N}\sum_{i=1}^{[N/2]}q^{i}\leq\frac{1}{N}\frac{q}{q-1}(q^{N/2}-1)\leq\frac{1}{2}\frac{q^{N}}{N},\]as $q^{N/2}-1\leq\frac{q-1}{2q}q^{N}$ for each $q$ and $N$ as in the hypotheses. 
Now, we have that:\[X=\left|\frac{q^{N}}{N}+\frac{1}{N}\sum_{d|N,d\neq N}\mu(\frac{N}{d})q^{d}\right|\]\[=\left|\frac{q^{N}}{N}-(-\frac{1}{N}\sum_{d|N,d\neq N}\mu(\frac{N}{d})q^{d})\right|\geq \left|\frac{q^{N}}{N}\right|-\left|1/N\sum_{d|N,d\neq N}\mu(\frac{N}{d})q^{d}\right|\geq\frac{q^{N}}{N}-\frac{1}{2}\frac{q^{N}}{N}=\frac{1}{2}\frac{q^{N}}{N},\]as a consequence of our previous estimate. To prove the other inequality we use an analogue of the factorization of the polynomial $T^{m}-1\in \mathbb{Q}[T]$ in cyclotomic polynomials with degree dividing $m$:\[T^{q^{N}}-T=\prod_{d|N}\phi_{d}(T)\]where $\phi_{d}(T)\in\mathbb{F}_{q}[T]$ is the product of the irreducible, monic polynomials with degree $d$. If we call $X_{d}$ the number of these ones, we have:\[\deg_{T}(\prod_{d|N}\phi_{d}(T))=NX+\sum_{d|N,d\neq N}dX_{d}=\deg_{T}(T^{q^{N}}-T).\]In particular, we have:\[X\leq\frac{q^{N}}{N}.\]
\end{proof}
We remark that an immediate consequence of Proposition 3 is that the set of Drinfeld modules which are RV($r$) is empty if $r>1$.\\\\ 
We state now a key lemma, of primary importance for our argument, as we will see. This is the \textbf{Siegel Lemma}, and its proof is contained in \cite{Denis}:
\begin{lem}
Let $a_{j,i}$ ($1\leq i\leq N$, $1\leq j\leq M$) be elements of $\ov{k}$ generating a finite algebraic extension $\widetilde{k}/k$ having degree $D$. We assume that $N>MD$. Then there exist $x_{1}, ..., x_{N}\in A$, not all $0$, such that:\[\sum_{1\leq i\leq N}x_{i}a_{j,i}=0\]for each $1\leq j\leq M$, and such that\[\deg_{T}(x_{i})\leq\frac{D}{N-MD}\sum_{1\leq j\leq M}h(a_{j,1}, ..., a_{j,N})\]for each $1\leq i\leq N$.
\end{lem}
\begin{lem}
Let $x\in\mathbb{D}(\ov{k})_{NT}$ with separable degree $D'_{sep.}$ over $k(\Phi)$, and let $\sigma_{1}, ..., \sigma_{D'_{sep.}}$ be the different embeddings of $k(\Phi)(x)$ in its algebraic closure in $\ov{k}$, fixing $k(\Phi)$. 
\begin{enumerate}
\item For each pair $(a,b)\in A^{2}$ such that $a/b\notin\mathbb{F}_{q}$, we have that:\[\sigma_{i}(\Phi(a)(x))\neq\sigma_{j}(\Phi(b)(x))\]for each pair $(i,j)\in\{1, ..., D'_{sep.}\}^{2}$.
	\item Let \textbf{M} be a subset of $A$ whose elements are pairwise coprime. Suppose that for each $a\in$ \textbf{M}, there exist $i\neq j$ in $\{1, ..., D'_{sep.}\}$ such that $\sigma_{i}(\Phi(a)(x))=\sigma_{j}(\Phi(a)(x))$. Then, the number of elements of \textbf{M} is less than $\log{D'_{sep.}}/\log{2}$.
\end{enumerate}
\end{lem}
\begin{proof}
\begin{enumerate}
	\item We consider without loss of generality that $\{\sigma_{1}, ..., \sigma_{D'_{sep.}}\}\subseteq Aut(k_{x}/k(\Phi))$, where $k_{x}$ is the normal closure of $k(\Phi)(x)$ in $\ov{k}$. If $\sigma_{i}(\Phi(a)(x))=\sigma_{j}(\Phi(b)(x))$ for some pair $(i,j)\in\{1, ..., D'_{sep.}\}^{2}$ and some $a$ and $b$ such that $a/b\notin\mathbb{F}_{q}$, $\Phi(a)(x)$ and $\Phi(b)(x)$ are conjugated over $k(\Phi)$, hence there exists $\sigma\in Aut(k_{x}/k(\Phi))$ 
	such that $\sigma(\Phi(a)(x))=\Phi(b)(x)$. As $Aut(k_{x}/k(\Phi))$ is a finite group, there exists $\mu\in \mathbb{N}\setminus\{0\}$ such that $\sigma^{\mu}=id_{k_{x}}$. We thus have that:\[\Phi(a^{\mu})(x)=\sigma^{\mu}(\Phi(a^{\mu})(x))=\Phi(a^{\mu-1})(\sigma^{\mu}(\Phi(a)(x)))=\Phi(a^{\mu-1})(\sigma^{\mu-1}(\Phi(b)(x)))\]\[=\Phi(a^{\mu-2})(\sigma^{\mu-1}\Phi(a)(\Phi(b)(x)))=\Phi(a^{\mu-2})(\sigma^{\mu-2}\Phi(b^{2})(x))=...=\Phi(b^{\mu})(x).\]Hence $\Phi(a^{\mu}-b^{\mu})(x)=0$. Since $x$ is not a torsion point, it follows that $a^{\mu}=b^{\mu}$, hence $a/b\in \mathbb{F}_{q}$. This contradicts the hypothesis.
	\item We take $a\in A$ and $j$ between $1$ and $D'_{sep.}$. Let:\[I(a,j)=\{i\in\{1, ..., D'_{sep.}\}/\sigma_{i}(\Phi(a)(x))=\sigma_{j}(\Phi(a)(x))\}.\]We have the following properties:
	
\begin{enumerate}
	\item $|I(a,j)|=|I(a,i)|$ for each pair $(i,j)\in\{1, ..., D'_{sep.}\}^{2}$ and two different sets of this form are disjoint.
	\item If $a$ and $b$ are coprime, $|I(a,i)\cap I(b,j)|\leq1$.
	\item If $a$ and $b$ are coprime, $|I(ab,j)|\geq |I(a,j)||I(b,j)|$.
\end{enumerate}
We start by proving the first point. If $i=j$ the statement is obvious. Let us assume $i\neq j$. We remark that for every $r\in I(a,i)$ there exists a unique $s\in \{1, ..., D'_{sep.}\}$ such that $\sigma_{r}^{-1}\sigma_{i}=\sigma_{s}^{-1}\sigma_{j}$. As $\sigma_{r}^{-1}\sigma_{i}(\Phi(a)(x))=\Phi(a)(x)=\sigma_{s}^{-1}\sigma_{j}(\Phi(a)(x))$, it is clear that $s\in I(a,j)$. The function $I(a,i)\to I(a,j)$ we have just constructed is moreover bijective: indeed, its inverse is defined in the same way in the opposite direction on the whole set $I(a,j)$. This shows that $|I(a,i)|=|I(a,j)|$. Lastly, if $I(a,i)\cap I(a,j)\neq \emptyset$, this immediately implies that $\sigma_{i}$ and $\sigma_{j}$ coincide on $k(\Phi(a)(x))$. 
The map $I(a,i)\to I(a,j)$ defined above takes then values in $I(a,i)$: indeed, for every $r\in I(a,i)$, we have that $\sigma_{r}(\Phi(a)(x))=\sigma_{i}(\Phi(a)(x))=\sigma_{j}(\Phi(a)(x))=\sigma_{s}(\Phi(a)(x))$, hence $\sigma_{i}^{-1}\sigma_{s}(\Phi(a)(x))=\Phi(a)(x)$, which means that $s\in I(a,i)$. It follows that $I(a,j)\subset I(a,i)$, hence $I(a,j)=I(a,i)$ since these finite sets have the same number of elements.\\\\
We now prove the second point: if $l,m\in I(a,i)\cap I(b,j)$, $\sigma_{m}(\Phi(b)(x))=\sigma_{l}(\Phi(b)(x))$ and $\sigma_{m}(\Phi(a)(x))=\sigma_{l}(\Phi(a)(x))$, so by Bachet-Bézout Theorem, $\sigma_{m}(\Phi((a,b))(x))=\sigma_{l}(\Phi((a,b))(x))$ (where the notation $(a,b)$ is to indicate the greatest common divisor of $a$ and $b$ in $A$), and therefore, as $a$ and $b$ are coprimes, $\sigma_{m}(x)=\sigma_{l}(x)$, so $m=l$.\\\\ 
In order to prove the third point, we first notice that the following inequality holds: $|I(ab,j)|\geq|\cup_{i\in I(a,j)}I(b,i)|$. Indeed, let $i\in I(a,j)$ and $l\in I(b,i)$. We have $\sigma_{l}(\Phi(b)(x))=\sigma_{i}(\Phi(b)(x))$, so $\sigma_{l}(\Phi(ab)(x))=\sigma_{i}(\Phi(ab)(x))=\sigma_{j}(\Phi(ab)(x))$. This shows that:\[\cup_{i\in I(a,j)}I(b,i)\subset I(ab,j),\]hence the above inequality. The two previous points now imply the point c.\\\\
Now, if we take $\textbf{M}$ as in the hypotheses, it follows that, for each $a\in \textbf{M}$ we have $|I(a,i)|\geq 2$ for each $i\in \{1, ..., D'_{sep.}\}$. Therefore:\[2^{|\textbf{M}|}\leq\prod_{a\in\textbf{M}}|I(a,i)|\leq|I(\prod_{a\in\textbf{M}}a,i)|\leq D'_{sep.}\]and:\[|\textbf{M}|\leq\frac{\log{D'_{sep.}}}{\log{2}}.\]
\end{enumerate}
\end{proof}
\section{Proof of Theorem 2}
We consider from now on a Drinfeld module $\mathbb{D}=(\mathbb{G}_{a},\Phi)$ which is RV($r$) for $r\in ]0,1]$. We first prove Theorem 2 assuming the RV($r$) hypothesis, then we will complete the proof to the RV($r$)$^{*}$ case by a slight modification of the last passages.\\\\
\textbf{Note:} From now on we let $x$ be an element of $\mathbb{D}(\ov{k})_{NT}$. If $D<q^{q+d+1}$, then Lemma 2 yields Theorem 2. Indeed, as it is easy to check, by Lemma 2 we have $\widehat{h}(x)\geq \frac{1}{c_{2}^{(q^{q+d+1}-1)^{2}}}$ and as $C_{0}\leq c_{2}^{-(q^{q+d+1}-1)^{2}}$, 
this yields $\widehat{h}(x)\geq C_{0}\geq C$ for all $C$ as in Theorem 2. Now, as $r\leq 1$, it follows that $dh(\Phi)c(\Phi)/r$ is always $\geq 1$. Thus $\mu\leq \kappa$, where $\mu$ and $\kappa$ are as in Theorem 2. We then conclude that $\frac{(\log\log_{+}D)^{\mu}}{D(\log_{+}D)^{\kappa}}\leq 1$ for all $D<q^{q+d+1}$ (where $\log_{+}(\cdot)$ and $\log\log_{+}(\cdot)$ take values $\geq 1$). Hence in what follows we will assume $D\geq q^{q+d+1}$. This will considerably ease many of the technicalities in our computations, as we will see.\\\\
To prove Theorem 2, we will argue by contradiction. We start therefore by assuming the following hypothesis which we want to contradict:
\begin{hyp}
We have:\[\widehat{h}(x)<C_{0}\frac{(\log\log{D})^{\mu}}{D{D}_{p.i.}^{\lambda}(\log{D})^{\kappa}}\]where $C_{0},\mu,\lambda,\kappa$ are defined as in Theorem 2. 
\end{hyp}
We will proceed by the following steps.

\begin{enumerate}
	\item We build an \textsl{auxiliary polynomial} with coefficients in $k(\Phi)$, vanishing with a certain multiplicity $t$ in $x$. By Siegel Lemma we will be able to bound the coefficients in some explicit way.
	\item We show that for some specific $h<t$, the auxiliary polynomial vanishes at $\Phi(l)(x)$ with multiplicity at least $h$, for every $l\in A\setminus \mathbb{F}_{q}$ monic irreducible of some specific degree satisfying the RV condition. The proof of this fact is the real heart of the whole section and will be accomplished arguing again by contradiction. Assuming that there is an $l\in S(A)$ as above and such that our polynomial vanishes at $\Phi(l)(x)$ with multiplicity $h'<h$, we show that this contradicts Siegel's Lemma.
	\item We will thus have that the sum of multiplicities of the roots $\Phi(l)(x)$ of our auxiliary polynomial for all $l\in S(A)$ satisfying the previous conditions (plus the case $l=1$), is at least $h$ times the cardinality of such a subset of $S(A)$. This will imply by a suitable choice of $h$ and $t$ that this number exceeds the degree of the auxiliary polynomial, yielding a contradiction. 
\end{enumerate}

\textbf{For the rest of the whole section we will assume Hypothesis 1.}

\subsection{Step 1 - Construction of the auxiliary polynomial}
\begin{de}
Given a polynomial $f(X)\in K[X]$ with $K$ any field of characteristic $p$, we call \textbf{hyperderivative} of $f$ of order $h$ the polynomial $d^{(h)}f(X)\in K[X]$ obtained as the coefficient of the term $H^{h}$ of $f(X+H)\in K[X][H]$, for some new parameter $H$.
\end{de}
\begin{rem}
Let $A(X)\in \ov{k}[X]$. An element $x\in \ov{k}$ is a root of $A(X)$ of multiplicity at least $h\geq 1$ if and only if $d^{(h')}A(x)=0$ for each $h'=0, ..., h-1$.
\end{rem}
\begin{proof}
See \cite{D}, Remarque 1.3.4.
\end{proof}
We call $p^{e}$ the purely inseparable degree $D'_{p.i.}$ of $k(\Phi)(x)$ over $k(\Phi)$.\\\\
The following Proposition provides the explicit construction of the auxiliary polynomial we will use in our proof. 
\begin{prop}
Let $L,t',t\in \mathbb{N}$ such that:\[t'=tp^{e}\]and:\[L^{2}>tDc(\Phi).\]Let $N\in A$ be such that:\begin{equation}\deg_{T}(N)=\left[\frac{1}{d}\log{L}\right]+1.\label{eq:8} \end{equation}Then there exists a polynomial:\[G(X,Y)=\sum_{i=0}^{L-1}\sum_{j=0}^{L-1}p_{ij}X^{i}Y^{j} \in A[X,Y]\setminus\{0\}\]such that:\[G_{N}(X):=G(X,\Phi(N)(X))\in k(\Phi)[X]\]is not identically 0, vanishes at $x$ with multiplicity at least $t'$ and such that the coefficients $p_{ij}\in A$ of $G(X,Y)$ satisfy the following condition:\[\deg_{T}(p_{ij})\leq \frac{Dc(\Phi)}{L^{2}-tDc(\Phi)}\Sigma\]for each $0\leq i,j\leq L-1$, where $\Sigma$ is the sum of the heights of all vectors which are the lines of the coefficient matrix of the linear system:\begin{equation} d^{(hp^{e})}G_{N}(x)=0 \label{eq:n} \end{equation} for $h=0, ..., t-1$, whose unknowns are precisely the coefficients of $G(X,Y)$.
\end{prop}
\begin{proof}
Let us write:\[G(X,Y)=\sum_{i=0}^{L-1}\sum_{j=0}^{L-1}p_{ij}X^{i}Y^{j}.\]We choose an element $N\in A\setminus\{0\}$ such that (\ref{eq:8}) holds. Therefore, as $q^{d\deg_{T}(N)}>L-1$, it follows that $G_{N}$ is not identically $0$ in $k(\Phi)[X]$ as the algebraic variety of equation $Y=\Phi(N)(X)$ in $\mathcal{C}^{2}$ is not contained in the zero locus of $G(X,Y)$. Indeed:\[Y-\Phi(N)(X)\nmid G(X,Y)\]in $k(\Phi)[X,Y]$. 
Now, the requirement that $G_{N}(X)$ vanishes at $x$ with order $t'$ means that we have to take the coefficients of $G_{N}(X)$ in the space of solutions of the linear system of $L^{2}$ unknowns and $t'$ conditions, given by the vanishing of the hyperderivatives of $G_{N}(X)$ at $x$ with order less than $t'$. We now show that the number of such conditions may actually be taken to be at most $t$. Indeed, if $x$ is a root of $G_{N}(X)$ we have as a first condition that:\[G_{N}(x)=0.\]This is a linear equation of $L^{2}$ unknowns and implies that:\[\Delta(X)|G_{N}(X)\]where $\Delta(X)\in k(\Phi)[X]$ is the minimal polynomial of $x$ over $k(\Phi)$. As the purely inseparable degree of $x$ over $k(\Phi)$ is $D'_{p.i.}=p^{e}$, $x$ is a root of $\Delta(X)$ of order $p^{e}$. Therefore, it is also a root of $G_{N}(X)$ with at least the same multiplicity. The single vanishing condition of $G_{N}(X)$ at $x$ implies therefore that the other $p^{e}-1$ conditions:\[d^{(h)}G_{N}(x)=0\]with $h\leq p^{e}-1$, are satisfied as well. In the same way, intersecting such a space of solutions with the one given by the linear equation:\[d^{(p^{e})}G_{N}(x)=0\]means that:\[\Delta(X)|\frac{G_{N}(X)}{\Delta(X)}\]in $k(\Phi)[X]$ as, by Remark 1, such an intersection implies that $x$ is a root of $G_{N}(X)$ with order at least $p^{e}+1$, while the roots of $\Delta(X)$ have multiplicity $p^{e}$. As:\[\Delta(X)^{2}|G_{N}(X)\]one actually has that the $p^{e}-1$ conditions:\[d^{(h)}G_{N}(x)=0\]with $h=p^{e}+1, ..., 2p^{e}-1$ follow directly from the first one $d^{(p^{e})}G_{N}(x)=0$. Now, repeating the same passages for each condition $G_{N}(x)=0$, $d^{(p^{e})}G_{N}(x)=0$, ..., $d^{((t-1)p^{e})}G_{N}(x)=0$, the linear system we have actually to solve takes the shape:\[d^{(hp^{e})}G_{N}(x)=0\]for $h=0, ..., t-1$ and it is equivalent to that of the form $d^{(h)}G_{N}(x)=0$ for each $h=0, ..., t'-1$.

We therefore obtain a linear system with $t$ conditions, whose unknowns are the $L^{2}$ coefficients of $G$. Thus, if we set $L^{2}>Dc(\Phi)t$, where $Dc(\Phi)\geq[k(\Phi)(x):k]$, the conditions of Lemma 3 are satisfied. If we choose $\Sigma$ as in the hypothesis the statement is proved.
\end{proof}
\subsection{Step 2 - Vanishing of $G_{N}(X)$ with desired multiplicity at $\Phi(l)(x)$ for $l$ RV in $S(A)$}
From now on, we choose the parameters $L$, $t$ and $h$ as follows (where $c_{0}$ is as in Theorem 2)\footnote{Such choices of the parameters have been suggested by Hugues Bauchère (LMNO, Caen), \cite{13}.}: \begin{equation} L:=\left[c_{0}^{2}\frac{D\log{D}}{(\log\log{D})^{2}}p^{e}\right]+1; \label{eq:14} \end{equation} \begin{equation} t:=\left[c_{0}^{3}\frac{D\log{D}}{(\log\log{D})^{3}}p^{e}\right]; \label{eq:15} \end{equation} \begin{equation} h:=\left[c_{0}\frac{D}{(\log\log{D})^{2}}\right]. \label{eq:16}\end{equation} Let us take, as before:\[t'=tp^{e}.\]
\begin{lem}
With the above choice of the parameters we have:\[L^{2}-tDc(\Phi)\geq \frac{1}{2}L^{2}.\]In particular, the hypothesis of Proposition 4 is satisfied.
\end{lem}
\begin{proof}
We know that:\[c_{0}\geq c(\Phi).\]Now, we state that: $\frac{1}{2}L^{2}-tDc(\Phi)\geq 0$. Indeed, as:\[L^{2}\geq c_{0}^{4}\frac{D^{2}(\log{D})^{2}p^{2e}}{(\log\log{D})^{4}}\]and:\[tDc(\Phi)\leq c_{0}^{4}\frac{D^{2}\log{D}}{(\log\log{D})^{3}}p^{2e}\]we have that:\[\frac{1}{2}L^{2}-tDc(\Phi)\geq c_{0}^{4}D^{2}p^{2e}\log{D}\left(\frac{\frac{1}{2}\log{D}-\log\log{D}}{(\log\log{D})^{4}}\right).\]If $D\geq q^{q+d+1}$ the right-hand term of such an inequality is not negative if and only if:\[\frac{1}{2}\log{D}\geq \log\log{D}\]which is easy to see to be always verified when $D\geq q^{q+d+1}$. 
\end{proof}
By Proposition 4, we can construct a polynomial:\[G_{N}(X)=\sum_{i=0}^{L-1}\sum_{j=0}^{L-1}p_{ij}X^{i}(\Phi(N)(X))^{j}.\]By Remark 1 we can say that $x$ is a root of multiplicity at least $t'-hp^{e}$ of $d^{(hp^{e})}G_{N}(X)$ for $h=0, ..., t-1$. For a general $h\leq t-1$ we thus have the following decomposition:\[d^{(hp^{e})}G_{N}(X)=\Delta(X)^{t-h}R_{h}(X)\]where $\Delta(X)\in k(\Phi)[X]$ is the minimal polynomial of $x$ over $k(\Phi)$.\\\\
The goal of this subsection is to prove the following proposition.
\begin{prop}
Let $l\in S(A)$ satisfying the RV property and such that: \begin{equation} \deg_{T}(l):=h(\Phi)c(\Phi)\left[\frac{1}{r}\log\left(c_{0}^{4}\frac{(\log{D})^{3}p^{2e}}{\log\log{D}}\right)\right]. \label{eq:17} \end{equation}
Then we have:\[d^{(h'p^{e})}G_{N}(\Phi(l)(x))=0\]for each $0\leq h'\leq h-1$.
\end{prop}
We will argue by contradiction. Assuming the conclusion of Proposition 5 is false, we will follow three steps:
\begin{enumerate}
	\item We provide an upper bound for the logarithmic height of $d^{(h'p^{e})}G_{N}(\Phi(l)(x))$.
	\item We prove a lower bound of the same quantity, using our assumption that $l$ satisfies the RV condition.
	\item We show that these two inequalities yield a contradition.
\end{enumerate}
Given $v_{l}$ a place of $k$ associated to an irreducible element $l\in A\setminus\{0\}$, we write $w|v_{l}$ to say that a place $w$ extends $v_{l}$ to $k(\Phi)(x)$.
\begin{prop}
In order to prove Theorem 2, we may assume that for all $l\in S(A)$ satisfying the RV condition and for all $w|v_{l}$ we have $w(x)\geq 0$.
\end{prop}
\begin{proof}
We assume that there exists an $l\in S(A)$ which is RV and an extension $w_{0}|v_{l}$ such that $w_{0}(x)<0$. Let us write:\[\Phi(l)(x)=lx+\alpha_{1}x^{q}+...+\alpha_{d\deg_{T}(l)}x^{q^{d\deg_{T}(l)}}.\]The RV hypothesis on $l$ implies that $w_{0}(\alpha_{i})>0$ for each $i=0, ..., d\deg_{T}(l)-1$, while $w_{0}(\alpha_{d\deg_{T}(l)})=0$. Therefore, as $w_{0}(x)<0$, we have that:\[w_{0}(\alpha_{d\deg_{T}(l)}x^{q^{d\deg_{T}(l)}})=q^{d\deg_{T}(l)}w_{0}(x)< q^{i}w_{0}(x)< w_{0}(\alpha_{i})+w_{0}(x^{q^{i}})=w_{0}(\alpha_{i}x^{q^{i}})\]for each $i=0, ..., d\deg_{T}(l)-1$. The properties of a non-Archimedean valuation imply therefore that:\[w_{0}(\Phi(l)(x))=q^{d\deg_{T}(l)}w_{0}(x).\]Iterating until we replace $x$ by $\Phi(l^{n-1})(x)$, we get, for all $n\in \mathbb{N}$:\[w_{0}(\Phi(l^{n})(x))=q^{d\deg_{T}(l)n}w_{0}(x).\]We derive from this:\[h(\Phi(l^{n})(x))=\frac{1}{[k(\Phi)(x):k]}\sum_{w\texttt{ over }k(\Phi)(x)/k}n_{w}\max\{0, -w(\Phi(l^{n})(x))\}\]\[\geq \frac{n_{w_{0}}}{[k(\Phi)(x):k]}\max\{0, -w_{0}(\Phi(l^{n})(x))\}\]\[=\frac{n_{w_{0}}q^{d\deg_{T}(l)n}}{[k(\Phi)(x):k]}\max\{0, -w_{0}(x)\}.\]Since $-n_{w_{0}}w_{0}(x)\geq 1$, we immediately deduce that:\[\widehat{h}(x)=\lim_{n\to+\infty}q^{-d\deg_{T}(l)n}h(\Phi(l^{n})(x))\geq \frac{1}{Dc(\Phi)}\]since:\[[k(\Phi)(x):k]\leq Dc(\Phi)\](we recall that $D=[k(x):k]$). This immediately provides an even stronger statement than Theorem 2, so we can easily get rid of the assumption of $w_{0}(x)$ to be negative.
\end{proof}
\subsubsection{Upper bound for $h(d^{(h'p^{e})}G_{N}(\Phi(l)(x)))$}
\begin{prop}
Let $l\in S(A)$. For each integer $h'$ with $0\leq h'\leq t-1$, we have:\[h(d^{(h'p^{e})}G_{N}(\Phi(l)(x)))\leq h(p_{ij})+L[h(\Phi(l)(x))+h(\Phi(Nl)(x))]+\deg_{T}(N) h(\Phi)h'p^{e}.\]
\end{prop}
\begin{proof}
We recall that:\[G(X,Y)=\sum_{i=0}^{L-1}\sum_{j=0}^{L-1}p_{ij}X^{i}Y^{j}.\]By definition, the hyperderivative of $G_{N}(X)$ at $x$ having order $h'p^{e}$ is the coefficient of $H^{h'p^{e}}$, for a new indeterminate $H$, of the polynomial:\[G_{N}(X+H)=\]\[=\sum_{i=0}^{L-1}\sum_{j=0}^{L-1}p_{ij}(\sum_{a=0}^{i}\binom{i}{a}X^{i-a}H^{a})(\sum_{b=0}^{j}\binom{j}{b}(\Phi(N)(X))^{j-b}(\Phi(N)(H))^{b}).\]Let us write:\[\Phi(N)(H)=\sum_{s=0}^{d\deg_{T}(N)}\widetilde{a}_{s}H^{q^{s}},\]where $\widetilde{a}_{s}\in k(\Phi)$ for $s=0, ..., d\deg_{T}(N)$. 
The hyperderivative $d^{(h'p^{e})}G_{N}(X)$ is a sum of a certain number of terms, which is the number of all the possible ways to obtain the power $H^{h'p^{e}}$ in the above expression of $G_{N}(X+H)$. Each of such terms takes the following shape:\[\sum_{i=a}^{L-1}\sum_{j=b}^{L-1}p_{ij}\binom{i}{a}\binom{j}{b}\binom{b}{n_{0}, ..., n_{d\deg_{T}(N)}}X^{i-a}(\Phi(N)(X))^{j-b}\prod_{s=0}^{d\deg_{T}(N)}\widetilde{a}_{s}^{n_{s}}\]\[=\sum_{i=a}^{L-1}\sum_{j=b}^{L-1}p_{ij}\binom{i}{a}\binom{j}{j-b,n_{0}, ..., n_{d\deg_{T}(N)}}X^{i-a}(\Phi(N)(X))^{j-b}\prod_{s=0}^{d\deg_{T}(N)}\widetilde{a}_{s}^{n_{s}}\]for each pair $(a,b)$ and each $(d\deg_{T}(N)+1)-$tuple $\ov{n}=(n_{0}, ..., n_{d\deg_{T}(N)})\in \mathbb{N}^{d\deg_{T}(N)+1}$ such that:\[b=\sum_{s=0}^{d\deg_{T}(N)}n_{s}\]and\[h'p^{e}=a+\sum_{s=0}^{d\deg_{T}(N)}n_{s}q^{s}.\]We thus obtain that $0\leq a\leq h'p^{e}$ and $0\leq \sum_{s=0}^{d\deg_{T}(N)}n_{s}q^{s}=h'p^{e}-a$.

For each pair $(i,j)\in\{0, ..., L-1\}^{2}$ the coefficient associated to $p_{ij}$ in the linear system (\ref{eq:n}) introduced in Proposition 4 is then:\[\sum_{(a,b,\ov{n})\in\mathcal{I}(i,j,h')}\binom{i}{a}\binom{j}{j-b,n_{0}, ..., n_{d\deg_{T}(N)}}x^{i-a}(\Phi(N)(x))^{j-b}\prod_{s=0}^{d\deg_{T}(N)}\widetilde{a}_{s}^{n_{s}},\] where we define the set $\mathcal{I}(i,j,h')$ as follows:\[\mathcal{I}(i,j,h'):=\{(a,b,n_{0}, ..., n_{d\deg_{T}(N)})\in \mathbb{N}^{d\deg_{T}(N)+3}, 0\leq a\leq \min\{i,h'p^{e}\},\]\[,a+\sum_{s=0}^{d\deg_{T}(N)}n_{s}q^{s}=h'p^{e}, \sum_{s=0}^{d\deg_{T}(N)}n_{s}=b\}.\]The height of the $h'-$th line $L_{h'}$ of the system (\ref{eq:n}) is thus:\[h(L_{h'}):=h(\{\sum_{(a,b,\ov{n})\in\mathcal{I}(i,j,h')}\binom{i}{a}\binom{j}{j-b,n_{0}, ..., n_{d\deg_{T}(N)}}x^{i-a}(\Phi(N)(x))^{j-b}\prod_{s=0}^{d\deg_{T}(N)}\widetilde{a}_{s}^{n_{s}}\}_{(i,j)}).\]
The property (\ref{eq:3}) of Proposition 1 leads us to the following upper bound of $h(L_{h'})$:\[h(L_{h'})\leq h(\{x^{i-a}(\Phi(N)(x))^{j-b}\prod_{s=0}^{d\deg_{T}(N)}\widetilde{a}_{s}^{n_{s}}\}_{(i,j,a,b,\ov{n})})\]\[=\frac{1}{[k(\Phi)(x):k]}\sum_{v\texttt{ over }k(\Phi)(x)/k}n_{v}\max_{(i,j,a,b,\ov{n})}\{0,-v(x^{i-a}(\Phi(N)(x))^{j-b}\prod_{s=0}^{d\deg_{T}(N)}\widetilde{a}_{s}^{n_{s}})\}.\]In other words, for each choice of $i,j=0, ..., L-1$, $a=0, ..., \min\{i,h'p^{e}\}$ and $b$ such that $h'p^{e}=a+\sum_{s=0}^{d\deg_{T}(N)}n_{s}q^{s}$, we multiply $x^{i-a}(\Phi(N)(x))^{j-b}$ by each one of the elements $\prod_{s=0}^{d\deg_{T}(N)}\widetilde{a}_{s}^{n_{s}}$ appearing for each $\ov{n}$ such that $\sum_{s=0}^{d\deg_{T}(N)}n_{s}=b$ and $h'p^{e}=a+\sum_{s=0}^{d\deg_{T}(N)}n_{s}q^{s}$. Therefore, by calling $\ov{\alpha}:=\{x^{i-a}(\Phi(N)(x))^{j-b}\}_{(i,j,a,b,\ov{n})}$ (note that for each multi-index $(i,j,a,b)$ the entry $x^{i-a}(\Phi(N)(x))^{j-b}$ appears a number of times which is precisely the cardinality of the set of the $\ov{n}$ associated to $(i,j,a,b)$), and calling:\[\ov{\beta}:=\{(\prod_{s=0}^{d\deg_{T}(N)}\widetilde{a}_{s}^{n_{s}})\}_{(i,j,a,b,\ov{n})}\]we have:\[\ov{\alpha}*\ov{\beta}=\{x^{i-a}(\Phi(N)(x))^{j-b}\prod_{s=0}^{d\deg_{T}(N)}\widetilde{a}_{s}^{n_{s}}\}_{(i,j,a,b,\ov{n})},\]whose height we are analysing. The product law (\ref{eq:2}) provides thus the following inequality:\[h(L_{h'})\leq h(\{x^{i-a}(\Phi(N)(x))^{j-b}\})+h(\{\prod_{s=0}^{d\deg_{T}(N)}\widetilde{a}_{s}^{n_{s}}\}).\]By the properties of the logarithmic height, the first term is bounded as follows:\[h(\{x^{i-a}(\Phi(N)(x))^{j-b}\})\leq L[h(x)+h(\Phi(N)(x))].\]We search now for an upper bound of the second term too. Writing $N(T)=
\alpha_{0}+\alpha_{1}T+...+\alpha_{\deg_{T}(N)}T^{\deg_{T}(N)}\in A$, we have:\[\Phi(N(T))=N(\Phi(T))=\alpha_{0}+\alpha_{1}\Phi(T)+...+\alpha_{\deg_{T}(N)}\Phi(T)^{\deg_{T}(N)}.\]We now focus on the height of each monomial. For $0\leq\delta\leq\deg_{T}(N)$:\[\Phi(T)^{\delta}=\sum_{i=0}^{d\delta}(\sum_{\ov{j}\in\Delta_{\delta}(i)}\prod_{s=1}^{\delta}a_{j(s)}^{q^{\sum_{\nu=0}^{s-1}j(\nu)}})\tau^{i}\]where:\[\Delta_{\delta}(i):=\{(j(1), ..., j(\delta))\in \mathbb{N}^{\delta};\sum_{s=1}^{\delta}j(s)=i\}\]with $j(s)\in\{0, ..., d\}$, $j(0):=0$. 
Recalling that $\widetilde{a}_{i}$ is the coefficient of $\tau^{i}$ in the expression of $\Phi(N)$, we obtain, for all places $w$ of $k(\Phi)$:\[-w(\widetilde{a}_{i})\leq\max\{\sum_{s=1}^{\delta}-q^{\sum_{\nu=0}^{s-1}j(\nu)}w(a_{j(s)})\}\leq \delta q^{i}\max_{j=0, ..., d}\{-w(a_{j})\}.\]Therefore:\[h(\{\prod_{s=0}^{d\deg_{T}(N)}\widetilde{a}_{s}^{n_{s}}\}_{(i,j,a,b,\ov{n})})=\frac{1}{c(\Phi)}\sum_{w\texttt{ over }k(\Phi)/k}n_{w}\max\{0,-\sum_{s=0}^{d\deg_{T}(N)}n_{s}w(\widetilde{a}_{s})\}\]\[\leq\frac{1}{c(\Phi)}\sum_{w\texttt{ over }k(\Phi)/k}n_{w}\sum_{s=0}^{d\deg_{T}(N)}\max\{0,-n_{s}w(\widetilde{a}_{s})\}\]\[\leq \frac{1}{c(\Phi)}\sum_{w\texttt{ over }k(\Phi)/k}n_{w}\sum_{s=0}^{d\deg_{T}(N)}n_{s} \deg_{T}(N) q^{s}\max_{j=0, ..., d}\{0,-w(a_{j})\}\leq \deg_{T}(N) h(\Phi)h'p^{e}.\]
Finally, for each $h'=0, ..., t-1$: \begin{equation} h(L_{h'})\leq L[h(x)+h(\Phi(N)(x))]+\deg_{T}(N) h(\Phi)h'p^{e}. \label{eq:10} \end{equation} Now:\[d^{(h'p^{e})}G_{N}(\Phi(l)(x))\]\[=\sum_{i=0}^{L-1}\sum_{j=0}^{L-1}p_{ij}\sum_{(a,b,\ov{n})\in\mathcal{I}(i,j,h')}\binom{i}{a}\binom{j}{j-b,n_{0}, ..., n_{d\deg_{T}(N)}}(\Phi(l)(x))^{i-a}(\Phi(Nl)(x))^{j-b}\prod_{i=0}^{d\deg_{T}(N)}\widetilde{a}_{i}^{n_{i}}.\]From this formula and the previous computations (applied to $\Phi(l)(x)$ instead of $x$), we obtain the bound of the Proposition. 
\end{proof}
\subsubsection{Lower bound for $h(d^{(h'p^{e})}G_{N}(\Phi(l)(x)))$}
\begin{prop}
Let $l\in S(A)$ satisfying the RV condition and $h'\in \mathbb{N}$ with $0\leq h'\leq t-1$. If $d^{(h'p^{e})}G_{N}(\Phi(l)(x))\neq 0$, then we have:\[h(d^{(h'p^{e})}G_{N}(\Phi(l)(x)))\geq \deg_{T}(l)\frac{(t-h')}{c(\Phi)}.\]
\end{prop}
\begin{proof}
We call $\zeta:=N_{k(\Phi)(x)/k}(d^{(h'p^{e})}G_{N}(\Phi(l)(x)))$. Since $\zeta$ is the product of $[k(\Phi)(x):k]$ conjugates (possibly equal) of $d^{(h'p^{e})}G_{N}(\Phi(l)(x))$, we have
:\[h(\zeta)
\leq [k(\Phi)(x):k]h(d^{(h'p^{e})}G_{N}(\Phi(l)(x))).\]
Let $w$ be a place of $k(\Phi)(x)$ such that $w|v_{l}$, and let $v$ be its restriction to $k(\Phi)$. Denote by $\mathcal{O}_{v}\subset k(\Phi)$ the valuation ring of $v$. As $w(x)\geq 0$ for each $w|v_{l}$ (see Proposition 6), the minimal (monic) polynomial $\Delta(X)$ of $x$ over $k(\Phi)$ has coefficients in $\mathcal{O}_{v}$. We know that:\[d^{(h'p^{e})}G_{N}(\Phi(l)(x))=\Delta(\Phi(l)(x))^{t-h'}R_{h'}(\Phi(l)(x))\]where $G_{N}(X)$ and $R_{h'}(X)$ are in $\mathcal{O}_{v}[X]\setminus\{0\}$
(we remark that as $\Delta(X), \Phi(l)(X)\in \mathcal{O}_{v}[X]$, we have $R_{h'}(X)\in \mathcal{O}_{v}[X]$ too). Let us call $\mathfrak{l}$ the prime ideal of $\mathcal{O}_{v}$ dividing $l$ and corresponding to $v$. Using the RV hypothesis on $l$ we have:\[\Delta(\Phi(l)(x))\equiv\Delta(x^{q^{d\deg_{T}(l)}})\texttt{ mod }(\mathfrak{l}\mathcal{O}_{w}),\]where $\mathcal{O}_{w}$ is the valuation ring of $k(\Phi)(x)$ with respect to $w$. Let $f_{v}=[\mathcal{O}_{v}/v:A/l]$ be the inertia degree of $v$ over $v_{l}$. By the RV condition, we have $f_{v}=1$, hence $|\mathcal{O}_{v}/v|=|A/l|=q^{\deg_{T}(l)}$. It follows that the coefficients of $\Delta(X)$ are congruent to their powers to $q^{d\deg_{T}(l)}$ mod ($\mathfrak{l}$). For this reason, we have that:\[\Delta(x^{q^{d\deg_{T}(l)}})\equiv \Delta(x)^{q^{d\deg_{T}(l)}}\equiv 0\texttt{ mod }(\mathfrak{l}\mathcal{O}_{w}).\]If we call $e_{w|v}$ the ramification index of $w$ over $v$, we have therefore that $w(\Delta(\Phi(l)(x)))\geq \frac{e_{w|v}}{e_{w}}=\frac{1}{e_{v}}$, where $e_{v}$ is the ramification index of $v$ over $v_{l}$. 
So we can conclude that: \begin{equation} w(d^{(h'p^{e})}G_{N}(\Phi(l)(x)))\geq \frac{t-h'}{e_{v}} \label{eq:m} \end{equation} for each $w|v_{l}$. Now, assuming $d^{(h'p^{e})}G(\Phi(l)(x))\neq 0$, we have that $\zeta\neq 0$ and therefore $\zeta^{-1}$ exists. 
Therefore:\[h(\zeta)=h(\zeta^{-1})\geq\max\{0,-\deg_{T}(l)v_{l}(\zeta^{-1})\}=\max\{0,\deg_{T}(l)v_{l}(\zeta)\}.\]By (\ref{eq:m}) we have that $v_{l}(\zeta)=\sum_{w|v_{l}}n_{w}w(d^{(h'p^{e})}G_{N}(\Phi(l)(x)))\geq \sum_{w|v_{l}}n_{w}\frac{t-h'}{e_{v}}\geq \sum_{w|v_{l}}n_{w}\frac{t-h'}{c(\Phi)}=[k(\Phi)(x):k]\frac{(t-h')}{c(\Phi)}$. 
Therefore:\[h(\zeta)\geq [k(\Phi)(x):k]\deg_{T}(l)\frac{(t-h')}{c(\Phi)}.\]As:\[h(\zeta)\leq [k(\Phi)(x):k]h(d^{(h'p^{e})}G_{N}(\Phi(l)(x)))\]the statement follows.
\end{proof}
\subsubsection{Final contradiction}
In this subsection, we will finally complete the proof of Proposition 5. Let $l\in S(A)$ be as in this proposition. Arguing by contradiction, let us assume that there exists $0\leq h'\leq h-1$ such that $d^{(h'p^{e})}G_{N}(\Phi(l)(x))\neq 0$. By the inequalities showed in Proposition 7 and Proposition 8, we have that:\[\deg_{T}(l)\frac{(t-h')}{c(\Phi)}\leq h(p_{ij})+L[h(\Phi(l)(x))+h(\Phi(Nl)(x))]+\deg_{T}(N)h(\Phi)h'p^{e}.\]And, as $h'\leq h-1$, we easily conclude that: \begin{equation} \deg_{T}(l)\frac{(t-h)}{c(\Phi)}\leq h(p_{ij})+L[h(\Phi(l)(x))+h(\Phi(Nl)(x))]+\deg_{T}(N)h(\Phi)hp^{e}. \label{eq:11} \end{equation} We now show that for $c_{0}$ as in Theorem 2 the choice of the parameters $L$, $t$ and $h$ contradicts the inequality (\ref{eq:11}). By Proposition 4 and inequality (\ref{eq:10}), we have:\[h(p_{ij})\leq \frac{Dc(\Phi)}{L^{2}-tDc(\Phi)}\sum_{0\leq h\leq t-1}(L[h(x)+h(\Phi(N)(x))]+\deg_{T}(N)h(\Phi)hp^{e}).\]Hence, by Lemma 5 and Proposition 2, we obtain:\[h(p_{ij})\leq 2\frac{Dc(\Phi)}{L^{2}}\sum_{0\leq h\leq t-1}(L[\widehat{h}(x)+2\gamma(\Phi)+\widehat{h}(\Phi(N)(x))]+\deg_{T}(N)h(\Phi)hp^{e})\]\[\leq 2\frac{Dtc(\Phi)}{L^{2}}(2Lq^{d\deg_{T}(N)}\widehat{h}(x)+4(d+1)Lh(\Phi)+\deg_{T}(N)h(\Phi)p^{e}t/2).\]Condition (\ref{eq:8}) now provides\footnote{As $\log{L}\geq d$ (which is a consequence of the inequality $D\geq q^{q+d+1}$) one has that $\deg_{T}(N)\leq \frac{2}{d}\log{L}$.} the inequality:\[h(p_{i,j})\leq 2\frac{Dtc(\Phi)}{L^{2}}(2q^{d}L^{2}\widehat{h}(x)+4(d+1)Lh(\Phi)+\frac{1}{d}\log{L}h(\Phi)tp^{e})\] \begin{equation} =4q^{d}c(\Phi)Dt\widehat{h}(x)+\frac{8(d+1)h(\Phi)c(\Phi)}{L}Dt+\frac{2h(\Phi)c(\Phi)}{d}\frac{Dt^{2}\log{L}}{L^{2}}p^{e}. \label{eq:19} \end{equation} \\\\
Now, by (\ref{eq:19}) and (\ref{eq:11}):\[\deg_{T}(l)\frac{(t-h)}{c(\Phi)}< 4q^{d}c(\Phi)Dt\widehat{h}(x)+\frac{8(d+1)h(\Phi)c(\Phi)}{L}Dt+\frac{2h(\Phi)c(\Phi)}{d}\frac{Dt^{2}\log{L}}{L^{2}}p^{e}\]\[+L[2q^{d(\deg_{T}(N)+\deg_{T}(l))}\widehat{h}(x)+4(d+1)h(\Phi)]+\deg_{T}(N)h(\Phi)hp^{e}\]\[\leq 4q^{d}c(\Phi)Dt\widehat{h}(x)+\frac{8(d+1)h(\Phi)c(\Phi)}{L}Dt+\frac{2h(\Phi)c(\Phi)}{d}\frac{Dt^{2}\log{L}}{L^{2}}p^{e}\]\[+2q^{d\deg_{T}(l)}q^{d}L^{2}\widehat{h}(x)+4(d+1)h(\Phi)L+\frac{2}{d}\log{L}h(\Phi)hp^{e}.\]The choices (\ref{eq:14}), (\ref{eq:15}) and (\ref{eq:16}) imply that $h\leq t/2$. So:\[\deg_{T}(l)\frac{t}{c(\Phi)}\leq  8q^{d}c(\Phi)Dt\widehat{h}(x)+\frac{16(d+1)h(\Phi)c(\Phi)}{L}Dt+\frac{4h(\Phi)c(\Phi)}{d}\frac{Dt^{2}\log{L}}{L^{2}}p^{e}\]\[+4q^{d\deg_{T}(l)}q^{d}L^{2}\widehat{h}(x)+8(d+1)h(\Phi)L+\frac{4}{d}\log{L}h(\Phi)hp^{e}.\]Knowing that (see Lemma 5) $tDc(\Phi)<L^{2}$, we obtain that:\begin{equation} \deg_{T}(l)t<c_{4}(L^{2}q^{d\deg_{T}(l)}\widehat{h}(x)+h(\Phi)L+\frac{h(\Phi)c(\Phi)}{d}\frac{Dt^{2}}{L^{2}}p^{e}\log{L}+\frac{h(\Phi)}{d}hp^{e}\log{L}); \label{eq:20} \end{equation} where we put:\[c_{4}:=24q^{d}c(\Phi).\] \\

We now find a lower bound for $\deg_{T}(l)t$.\\\\
For each $a,b\in \mathbb{R}^{+}$ such that $a,b\geq 4$, we have that $[a][b]\geq \frac{1}{2}ab$. Therefore, we pose:\[\alpha:=h(\Phi)c(\Phi)\]and we remark that:\[c_{0}\geq q^{d}.\]Such a condition implies, as $D\geq q^{q+d+1}$, that $t,\deg_{T}(l)\geq 4$. Therefore:\[\deg_{T}(l)t\geq\frac{1}{2}\frac{\alpha}{r}(4\log{c_{0}}+3\log\log{D}+2\log{p^{e}}-\log\log\log{D})c_{0}^{3}\frac{D\log{D}}{(\log\log{D})^{3}}p^{e}\]\[\geq \frac{1}{2}\frac{\alpha}{r}(4\log{c_{0}}+2\log\log{D})c_{0}^{3}\frac{D\log{D}}{(\log\log{D})^{3}}p^{e}\] \begin{equation} \geq \frac{\alpha}{r}c_{0}^{3}\frac{D\log{D}}{(\log\log{D})^{2}}p^{e}. \label{eq:23} \end{equation}\\\\
We now prove that (\ref{eq:20}) cannot hold if we assume Hypothesis 1. This will follow from the following facts: \begin{equation} \deg_{T}(l)t\geq 4c_{4}h(\Phi)L; \label{eq:24} \end{equation} \begin{equation} \deg_{T}(l)t\geq 4c_{4}\frac{h(\Phi)c(\Phi)}{d}\frac{Dt^{2}}{L^{2}}p^{e}\log{L}; \label{eq:25} \end{equation} \begin{equation} \deg_{T}(l)t\geq 4c_{4}\frac{h(\Phi)}{d}hp^{e}\log{L}; \label{eq:26} \end{equation}\begin{equation}\deg_{T}(l)t\geq 4c_{4}L^{2}q^{d\deg_{T}(l)}\widehat{h}(x).\label{eq:101}\end{equation}Let us check (\ref{eq:24}) first. We have
: \begin{equation} c_{0}^{2}\frac{D\log{D}}{(\log\log{D})^{2}}p^{e}\leq L\leq 2c_{0}^{2}\frac{D\log{D}}{(\log\log{D})^{2}}p^{e} \label{eq:27} \end{equation} by (\ref{eq:14}) and the hypothesis that $D\geq q^{q+d+1}$. The inequality (\ref{eq:24}) is thus by (\ref{eq:23}) a consequence of the following one:\[\frac{\alpha}{r}c_{0}^{3}\frac{D\log{D}}{(\log\log{D})^{2}}p^{e}\geq 8c_{4}\alpha c_{0}^{2}\frac{D\log{D}}{(\log\log{D})^{2}}p^{e}\]which is true since:\[c_{0}\geq 8rc_{4}=192rq^{d}c(\Phi).\]The inequality (\ref{eq:25}) is on the other hand by (\ref{eq:23}) and (\ref{eq:27}) a consequence of the following one:\[\frac{\alpha}{r}c_{0}^{3}\frac{D\log{D}}{(\log\log{D})^{2}}p^{e}\geq 4c_{4}\frac{\alpha}{d}c_{0}^{6}\frac{D^{3}(\log{D})^{2}p^{2e}}{(\log\log{D})^{6}}\frac{(\log\log{D})^{4}}{c_{0}^{4}p^{2e}D^{2}(\log{D})^{2}}p^{e}\log{L}\]which follows from this condition:\[c_{0}\log{D}\geq \frac{4rc_{4}}{d}(2\log{c_{0}}+\log{D}+\log\log{D}-2\log\log\log{D}+\log{2}+\log{p^{e}})\]which is implied by the following inequality: \begin{equation} c_{0}\log{D}\geq\frac{4rc_{4}}{d}(2\log{c_{0}}+4\log{D}). \label{eq:29} \end{equation} Now, (\ref{eq:29}) follows from these two facts:\[c_{0}\geq \frac{32rc_{4}}{d}=\frac{768rq^{d}c(\Phi)}{d},\]and:\[c_{0}\log{D}\geq \frac{16rc_{4}}{d}\log{c_{0}}.\]These are a consequence of:\begin{equation}\frac{c_{0}}{\log{c_{0}}}\geq\frac{16rc_{4}}{d}=\frac{384rq^{d}c(\Phi)}{d}.\label{eq:100}\end{equation}To prove (\ref{eq:100}) we consider two cases. Suppose first that $r\leq \frac{1}{384}$. As the function $\frac{X}{\log{X}}$ increases for $X\geq e$, hence in particular for $X\geq 2q^{d}c(\Phi)^{2}$, we use the fact that:\[c_{0}\geq 2q^{d}c(\Phi)^{2};\]so that (\ref{eq:100}) is implied by the following one:\[\frac{2q^{d}c(\Phi)^{2}}{d+2\log{c(\Phi)}+\log{2}}\geq \frac{384rq^{d}c(\Phi)}{d}.\]This inequality is satisfied because $r\leq \frac{1}{384}$ and $\frac{2c(\Phi)}{d+2\log{c(\Phi)}+\log{2}}\geq \frac{1}{d}$. 
Let us now examine the case where:\[r>\frac{1}{384}.\]Since $c_{0}\geq 35000dq^{d}c(\Phi)^{2}$ and $r\leq 1$, 
there exists a real number $X_{0}\geq 91d$ such that:\[c_{0}=X_{0}384rq^{d}c(\Phi)^{2}.\]We now claim that:\[\frac{X_{0}c(\Phi)}{\log{X_{0}}+\log(384r)+d+2\log{c(\Phi)}}\geq \frac{1}{d},\]for all $d\geq 1$. Indeed, as $q\geq 2$ and $\log(384r)\leq \log_{2}{384}< 9$, this inequality follows by:\[\frac{X_{0}c(\Phi)}{\log{X_{0}}+9+d+2\log{c(\Phi)}}\geq \frac{1}{d},\]and such a fact is true for all $X_{0}\geq 91$, so in particular for $X_{0}\geq 91d$ as well, for $d\geq 1$. Thus:\[\frac{c_{0}}{\log{c_{0}}}\geq \frac{X_{0}384rq^{d}c(\Phi)^{2}}{\log{X_{0}}+9+d+2\log{c(\Phi)}}\geq \frac{384rq^{d}c(\Phi)}{d},\]which gives (\ref{eq:100}). This completes the proof of the inequality (\ref{eq:25}). 

By (\ref{eq:23}) and (\ref{eq:27}) we have that condition (\ref{eq:26}) is a consequence of the following:\[\frac{\alpha}{r}c_{0}^{3}\frac{D\log{D}}{(\log\log{D})^{2}}p^{e}\geq 4c_{4}\frac{\alpha}{d}c_{0}\frac{D}{(\log\log{D})^{2}}p^{e}(2\log{c_{0}}+4\log{D})\]and, therefore, of the following one: \begin{equation} c_{0}^{2}\log{D}\geq \frac{8rc_{4}}{d}(2\log{c_{0}}+4\log{D}). \label{eq:38} \end{equation} As (\ref{eq:38}) is implied by (\ref{eq:29}), we thus also have (\ref{eq:26}).\\
Lastly, (\ref{eq:101}) follows precisely by Hypothesis 1, which yields the contradiction we need to prove Proposition 5. Indeed, by Hypothesis 1 we have:\[\widehat{h}(x)<C_{0}\frac{(\log\log{D})^{2+\frac{d\alpha}{r}}}{D(p^{e})^{1+\frac{2d}{r}\alpha}(\log{D})^{1+\frac{3d}{r}\alpha}}\leq \frac{\alpha(\log\log{D})^{2+\frac{d}{r}\alpha}}{r384q^{d}c(\Phi)c_{0}^{\frac{4d}{r}\alpha+1}D(p^{e})^{1+\frac{2d}{r}\alpha}(\log{D})^{1+\frac{3d}{r}\alpha}}.\]By (\ref{eq:27}) and (\ref{eq:17}) one can now easily check that:\[4c_{4}L^{2}q^{d\deg_{T}(l)}\widehat{h}(x)<c_{0}^{3}\frac{D\log{D}}{(\log\log{D})^{2}}p^{e}\frac{\alpha}{r},\]which by (\ref{eq:23}) yields (\ref{eq:101}). We have therefore proved that inequality (\ref{eq:20}) cannot hold. This contradiction completes the proof of Proposition 5.
$\square$
\subsection{Step 3 - Counting zeroes of $G_{N}(X)$}
We have shown in subsection 3.2 that assuming Hypothesis 1 yields $G_{N}(\Phi(l)(x))=0$ for each $l$ which satisfies the RV property and $\deg_{T}(l)$ chosen as in (\ref{eq:17}). Now, by Galois Theory we know that for each of such $l$ all the conjugates of $\Phi(l)(x)$ over $k(\Phi)$ are also zeroes of $G_{N}(X)$ with the same multiplicity of $\Phi(l)(x)$. Using Lemma 4 we can now compute the number of zeroes of $G_{N}(X)$ with their multiplicity. As we assume the Drinfeld module to be RV($r$) the polynomial $G_{N}(X)$ turns out to have at least\[(\frac{q^{r\deg_{T}(l)}}{2r\deg_{T}(l)}-\frac{\log{D'_{sep.}}}{\log{2}})D'_{sep.}\]zeroes, with multiplicity at least $hp^{e}$, where $\deg_{T}(l)$ is defined as in (\ref{eq:17}). As $\log{D'_{sep.}}\leq \log{D_{sep.}}$ and $D=[k(x):k]\leq [k(\Phi)(x):k(\Phi)][k(\Phi):k]=D'c(\Phi)$, it follows that the sum of multiplicities of the roots of $G_{N}(X)$ is at least\[(\frac{q^{r\deg_{T}(l)}}{2r\deg_{T}(l)}-\frac{\log{D_{sep.}}}{\log{2}})\frac{D}{c(\Phi)}h.\]Knowing that:\[\deg_{X}(G_{N}(X))\leq 2(L-1)q^{d\deg_{T}(N)}< 2q^{d}L^{2}\]we will now prove that Hypothesis 1 is false by showing that: \begin{equation} (\frac{q^{r\deg_{T}(l)}}{2r\deg_{T}(l)}-\frac{\log{D_{sep.}}}{\log{2}})\frac{D}{c(\Phi)}h\geq 2q^{d}L^{2}> \deg_{X}(G_{N}(X)). \label{eq:40} \end{equation} Indeed, (\ref{eq:40}) would prove that the sum of multiplicities of the roots of $G_{N}(X)$ exceeds the degree, so $G_{N}(X)$ has to be identically $0$, which would contradict Proposition 4.
\begin{prop}
As $c_{0}\geq 2q$, we have that: \begin{equation} \frac{q^{r\deg_{T}(l)}}{2r\deg_{T}(l)}\geq 2\frac{\log{D_{sep.}}}{\log{2}}. \label{eq:41} \end{equation} 
\end{prop}
\begin{proof}
As (\ref{eq:17}) implies that:\[\alpha\log\left(c_{0}^{4}\frac{(\log{D})^{3}p^{2e}}{\log\log{D}}\right)\geq r\deg_{T}(l)\geq r\alpha\left(\frac{1}{r}\log\left(c_{0}^{4}\frac{(\log{D})^{3}p^{2e}}{\log\log{D}}\right)-1\right)\]it follows that:\[q^{r\deg_{T}(l)}\geq \frac{\left(c_{0}^{4}\frac{(\log{D})^{3}p^{2e}}{\log\log{D}}\right)^{\alpha}}{q^{r\alpha}}.\]Therefore: \begin{equation} \frac{q^{r\deg_{T}(l)}}{2r\deg_{T}(l)}\geq \frac{\left(c_{0}^{4}\frac{(\log{D})^{3}p^{2e}}{\log\log{D}}\right)^{\alpha}}{q^{r\alpha}2\alpha(4\log{c_{0}}+3\log\log{D}+2\log{p^{e}}-\log\log\log{D})}. \label{eq:42} \end{equation} Condition (\ref{eq:41}) will thus be a consequence of the following one:\[\frac{\left(c_{0}^{4}\frac{(\log{D})^{3}p^{2e}}{\log\log{D}}\right)^{\alpha}}{q^{r\alpha}2\alpha(4\log{c_{0}}+3\log\log{D}+2\log{p^{e}})}\geq 2\frac{\log{D_{sep.}}}{\log{2}}.\]We thus have to show that:\[c_{0}^{4\alpha}(\log{D})^{3\alpha-1}p^{2\alpha e}\geq \frac{q^{r\alpha}4\alpha}{\log{2}}(\log\log{D})^{\alpha}(4\log{c_{0}}+3\log\log{D}+2\log{p^{e}}).\]Such an inequality is a consequence of the following three conditions: \begin{equation} c_{0}^{4\alpha}(\log{D})^{3\alpha-1}p^{2\alpha e}\geq \frac{48q^{r\alpha}\alpha}{\log{2}}(\log\log{D})^{\alpha}\log{c_{0}}, \label{eq:43} \end{equation} \begin{equation} c_{0}^{4\alpha}(\log{D})^{3\alpha-1}p^{2\alpha e} \geq \frac{36q^{r\alpha}\alpha}{\log{2}}(\log\log{D})^{\alpha+1}; \label{eq:44} \end{equation} \begin{equation} c_{0}^{4\alpha}(\log{D})^{3\alpha-1}p^{2\alpha e} \geq \frac{24q^{r\alpha}\alpha}{\log{2}}(\log\log{D})^{\alpha}\log{p^{e}}. \label{eq:45} \end{equation} As we are assuming that $D\geq q^{q+d+1}$, (\ref{eq:43}) is satisfied if: \begin{equation} \frac{c_{0}^{4\alpha}}{\log{c_{0}}}\geq \frac{48q^{r\alpha}\alpha}{\log{2}}. \label{eq:46} \end{equation} Now, taking $c_{0}$ as in Theorem 2 and, therefore, such that:\[c_{0}\geq 2q,\]we see that (\ref{eq:46}) is satisfied. Indeed, we have that:\[\frac{2^{4\alpha}q^{3}}{\log{2}+1}\geq \frac{48\alpha}{\log{2}}\]and:\[q^{4\alpha-3}\geq q^{r\alpha}\]which is always true for each $\alpha \geq 1$ and $r\leq 1$.\\\\
Now, (\ref{eq:44}) follows from this inequality:\[c_{0}^{4\alpha}p^{\alpha e}\geq \frac{36q^{r\alpha}\alpha}{\log{2}}\]which is satisfied by (\ref{eq:46}). (\ref{eq:45}) directly follows from (\ref{eq:46}). 
We thus proved (\ref{eq:41}).
\end{proof}


\textbf{Proof of Theorem 2}\\\\
Recall that:\[c_{0}=35000 d\alpha^{3}q^{d+r\alpha}.\]If we call:\[A:=35000\]we see that the following three inequalities hold:\[\frac{A}{\log{A}}\geq 2304\]\[\frac{Ad}{d+\log{d}}\geq 2304\]\[\frac{A\alpha}{3\log{\alpha}+\alpha}\geq 2304.\]As:\[c_{0}=Ad\alpha^{3}q^{d+r\alpha},\]we have by the three inequalities above that:
\[\frac{c_{0}}{\log{c_{0}}}\geq \left(\frac{\log{A}+\log{d}+d+3\log{\alpha}+\alpha}{Ad\alpha^{3}q^{d+r\alpha}}\right)^{-1}\]\[\geq \left(\frac{1}{2304d\alpha^{3}q^{d+r\alpha}}+\frac{1}{2304\alpha^{3}q^{d+r\alpha}}+\frac{1}{2304d\alpha^{2}q^{d+r\alpha}}\right)^{-1}\geq \frac{2304\alpha^{2}q^{d+r\alpha}}{3}=768\alpha^{2}q^{d+r\alpha}.\]We now see, by calling:\[X:=\log{D}\]and remembering that $D\geq q^{q+d+1}$ that:\[\frac{768q^{d+r\alpha}\alpha^{2}(\log{X})^{\alpha-2}\log{c_{0}}}{c_{0}^{4\alpha-3}X^{3\alpha-2}}\leq 1,\]\[\frac{384q^{d+r\alpha}\alpha^{2}(\log{X})^{\alpha-2}}{c_{0}^{4\alpha-3}X^{3(\alpha-1)}}\leq 1,\]\[\frac{576q^{d+r\alpha}\alpha^{2}(\log{X})^{\alpha-1}}{c_{0}^{4\alpha-3}X^{3\alpha-2}}\leq 1.\]
Hence we have:
\[\frac{64q^{d+r\alpha}\alpha^{2}((4\log{c_{0}}+2\log{p^{e}})(\log{X})^{\alpha-2}+3(\log{X})^{\alpha-1})}{c_{0}^{4\alpha-3}X^{3\alpha-2}p^{2(\alpha-1)e}}\leq 1.\]Therefore, since $c(\Phi)\leq \alpha$:\[\frac{c_{0}^{4\alpha-3}(\log{D})^{3\alpha-2}p^{2(\alpha-1)e}}{64q^{r\alpha+d}\alpha(4\log{c_{0}}+3\log\log{D}+2\log{p^{e}})c(\Phi)(\log\log{D})^{\alpha-2}}\geq 1.\]By (\ref{eq:42}) and by the fact that:\[h\geq \frac{1}{2}c_{0}\frac{D}{(\log\log{D})^{2}},\]and:\[L\leq 2c_{0}^{2}\frac{D\log{D}}{(\log\log{D})^{2}}p^{e},\](which are a consequence of the hypotheses $D\geq q^{q+d+1}$ and $c_{0}\geq 1$) we have that:\[\frac{q^{r\deg_{T}(l)}}{4r\deg_{T}(l)}\frac{D}{c(\Phi)}h\]\[\geq \frac{\left(c_{0}^{4}\frac{(\log{D})^{3}p^{2e}}{\log\log{D}}\right)^{\alpha}}{q^{r\alpha}4\alpha(4\log{c_{0}}+3\log\log{D}+2\log{p^{e}}-\log\log\log{D})}\frac{D}{c(\Phi)}\frac{1}{2}c_{0}\frac{D}{(\log\log{D})^{2}}\]\[\geq 2q^{d}4c_{0}^{4}\frac{D^{2}(\log{D})^{2}}{(\log\log{D})^{4}}p^{2e}\geq 2q^{d}L^{2}.\]
By Proposition 9 it is now easy to see that (\ref{eq:40}) immediately follows. As we have seen, this is a contradiction. This means that Hypothesis 1 is false, hence the first part of Theorem 2 is proved under the RV($r$) hypothesis.\\\\

We now prove the second statement of Theorem 2, involving the RV($r$)$^{*}$ hypothesis.\\\\
We thus assume that the condition RV($r$)$^{*}$ is satisfied by the Drinfeld module $\mathbb{D}=(\mathbb{G}_{a}, \Phi)$. 
In our new situation, the condition RV($r$)$^{*}$ is not anymore sufficient to ensure the existence of all elements of $P_{\deg_{T}(l)}(A)$ with the desired degree in $T$. We will have a sufficiently large number of elements of $P_{\deg_{T}(l)}(A)$ only assuming a value of $\deg_{T}(l)$ large enough. We therefore modify the previous proof as follows. Let $N(\Phi)$ be an integer satisfying the conditions of Definition 3. We choose a positive integer $D_{\Phi}$ with $D_{\Phi}>q^{q+d+1}$ such that, for all $D\geq D_{\Phi}$, we have:\[h(\Phi)c(\Phi)\left[\frac{1}{r}\log\left(c_{0}^{4}\frac{(\log{D})^{3}}{\log\log{D}}\right)\right]\geq N(\Phi).\]We have now two cases. If $D\geq D_{\Phi}$, then we repeat exactly the same proof as before (with the same choice for the parameters $L$, $t$, $h$, $\deg_{T}(l)$) and we obtain the bound of Theorem 2 with $C=C_{0}$. If now $D<D_{\Phi}$, then we obtain, by Lemma 2, the lower bound of the theorem with:\[C=\min\{q^{-5d(2(d+1)h(\Phi)+1)((D_{\Phi}-1)c(\Phi))^{2}},\frac{h(\Phi)}{384rq^{d}c_{0}^{\frac{4h(\Phi)c(\Phi) d}{r}+1}}\}\leq C_{0}.\]In both cases, we thus get the estimate of Theorem 2. 

\subsection{Separable case}
A little improvement of the value of $c_{0}$ may be obtained if we restrict to the hypothesis that $x$ is \textbf{separable}. More precisely we have the following statement. 
\begin{thm}
Let $\mathbb{D}=(\mathbb{G}_{a},\Phi)$ be a Drinfeld module defined over $\ov{k}$ satisfying the hypothesis $RV(r)$ or $RV(r)^{*}$. Let:\[c_{0}:=6500dh(\Phi)^{3}c(\Phi)^{3}q^{d+rh(\Phi)c(\Phi)}\]and:\[C_{0}:=\min\{q^{-5d(2(d+1)h(\Phi)+1)((q^{q+d+1}-1)c(\Phi))^{2}},\frac{h(\Phi)}{768rq^{d}c_{0}^{1+\frac{4d}{r}h(\Phi)c(\Phi)}}\}.\]Then, there exists $C>0$ such that for all $x\in\mathbb{D}(\ov{k})_{NT}$ separable with degree $D$ over $k$, one has:\[\widehat{h}_{\mathbb{D}}(x)\geq C\frac{(\log\log_{+}{D})^{2+\frac{d}{r}h(\Phi)c(\Phi)}}{D(\log_{+}{D})^{1+\frac{2d}{r}h(\Phi)c(\Phi)}}\]where:\[C=C_{0}\texttt{   under the hypothesis   }RV(r)\]while\[0<C\leq C_{0}\texttt{   under the hypothesis   }RV(r)^{*}.\]

\end{thm}
\begin{proof}
The proof repeats exactly the same steps as in the inseparable case, just assuming $D_{p.i.}=1$. We send the reader to \cite{D} for the explicit passages.
\end{proof}
This result contains L. Denis' result (see Theorem 1) about Carlitz modules: 
\begin{cor}
Under the same hypotheses of Theorem 3, for $\mathbb{D}$ taken as the Carlitz module (which is RV($1$)) one finds the estimate of L. Denis (Theorem 1).
\end{cor}

\section{Appendix: Drinfeld modules and supersingular reduction primes}
In this section we concretely produce examples of Drinfeld modules satisfying RV($r, c_{1}$)$^{*}$ properties, showing in particular (see Theorem 6) that all CM Drinfeld modules with coefficients in $k$ and having rank $1$ or a prime number different from the field characteristic essentially belong to one of these classes. 
Chantal David already showed remarkably (see \cite{C. David} Theorem 1.2) that "in average" a rank 2 Drinfeld module with coefficients in $k$ satisfies the RV($r,c_{q}$)$^{*}$ condition, with $r=1/d=1/2$ for $c_{q}>0$ a constant depending only on $q$. 
\begin{de}
Let $r\in]0,1]$, $c_{1}\in \mathbb{R}_{>0}$ and $\eta\in \mathbb{N}\setminus\{0\}$ and let $\mathbb{D}=(\mathbb{G}_{a},\Phi)$ be a Drinfeld module. We say that $\mathbb{D}$ is \textbf{RV$_{\eta}$($r,c_{1}$)$^{*}$} if there exists a positive integer $N(\Phi)$ only depending on the choice of $\mathbb{D}$, such that for each $N\in \mathbb{N}$ such that $N\geq N(\Phi)$ and $N\equiv 1$ mod ($\eta$), we have:\[|\{l\in P_{N}(A), l\texttt{ is }RV\}|\geq c_{1}\frac{q^{rN}}{N}.\]
\end{de}
As it is easy to see, a Drinfeld module RV($r,c_{1}$)$^{*}$ is also RV$_{\eta}$($r,c_{1}$)$^{*}$ for $\eta>1$ and the two classes coincide when $\eta=1$.\\\\
Given a Drinfeld module $\mathbb{D}=(\mathbb{G}_{a},\Phi)$, we call a monic irreducible element $p(T)\in$ $S(A)$ which satisfies the RV condition with respect to $\mathbb{D}$ a \textbf{supersingular reduction prime} of $\Phi$. Note that our definition of supersingular prime is stronger than the one which is commonly used, only requiring supersingular reduction of the chosen Drinfeld module at $p(T)$, while the RV property also claims that all primes over $p(T)$ in the field of coefficients have inertia degree 1 on $p(T)$. For this reason we will focus only on Drinfeld modules defined over $k$, so that in such a setting our special notion of supersingular reduction prime clearly coincides with the common one. Thus, there will be no more need to make a distinction between the two definitions.\\\\
Let $\mathbb{D}=(\mathbb{G}_{a},\Phi)$ be a Drinfeld module of any \textbf{characteristic}\footnote{See \cite{Goss}, Definition 4.4.1.}, defined over $k$. We set:\[End_{k}(\Phi):=\{P(\tau)\in k\{\tau\}, \Phi(a)P=P\Phi(a),\forall a\in A\}.\]This is an $A-$module with respect to the action of $\Phi$ and a subring of $k\{\tau\}$ as well. One can see that this is a free $A-$module.
\begin{lem}
Let $\mathbb{D}$ be a Drinfeld module of characteristic 0, defined over $k$, of rank $d$. The rank of the $A-$module $End_{k}(\Phi)$ divides $d$.
\end{lem}
\begin{proof}
See \cite{D}, Lemma 1.4.3.
\end{proof}
\begin{de}
A Drinfeld module $\mathbb{D}=(\mathbb{G}_{a},\Phi)$ defined over $k$ is called \textbf{CM} or \textbf{with complex multiplication} if the rank of $End_{k}(\Phi)$ as an $A-$module is $d$.
\end{de}
We remark that every Drinfeld module with rank 1 has complex multiplication.
\subsection{Extending $\Phi$ to $End_{k}(\Phi)$}
Let $\mathbb{D}=(\mathbb{G}_{a},\Phi)$ be a Drinfeld module with rank $d$ and characteristic 0, defined over $k$. 
Reduction of $\Phi(T)$ modulo $p(T)$ gives for every $p(T)\in$ $S(A)$, except possibly finitely many, a \textbf{reduced} Drinfeld module $\mathbb{D}_{p}:=(\mathbb{G}_{a},\Phi^{v_{p}})$, where, by calling $s:=\deg_{T}(p(T))$:\[\Phi^{v_{p}}:A\to \mathbb{F}_{q^{s}}\{\tau\}\]is the $\mathbb{F}_{q}-$algebra homomorphism defined by the association:\[T\mapsto (\Phi(T))^{v_{p}},\]where $(\Phi(T))^{v_{p}}$ is the reduction modulo $p(T)$ of the twisted polynomial $\Phi(T)$ and has still degree $d$ (see \cite{Goss}, Definition 4.10.1). 
One can moreover see (\cite{Goss}, chapter 4) that there exists an injective ring homomorphism:\[\Phi^{v_{p}}:A\hookrightarrow End_{\mathbb{F}_{q^{s}}}(\Phi^{v_{p}}).\]
It can be extended in the following way:\[\widetilde{\Phi^{v_{p}}}:End_{k}(\Phi)\hookrightarrow End_{\mathbb{F}_{q^{s}}}(\Phi^{v_{p}})\]\[P(\tau)\mapsto P(\tau)^{v_{p}}\]where $P(\tau)^{v_{p}}$ is obtained by reducing modulo $p(T)$ the coefficients of $P(\tau)$. 
We know (see \cite{Goss}, Proposition 4.7.13) that any isogeny between two Drinfeld modules divides an element of $A\setminus\{0\}$. We thus tensorize over $A$ with $k$ the category of Drinfeld modules and isogenies. One therefore extends in a natural fashion the algebra homomorphism $\Phi^{v_{p}}$ to $k$. This provides a field embedding $End_{k}(\Phi)\otimes_{A}k\hookrightarrow End_{\mathbb{F}_{q^{s}}}(\Phi^{v_{p}})\otimes_{A}k$. 
We now call:\[D_{p}:=End_{\mathbb{F}_{q^{s}}}(\Phi^{v_{p}})\otimes_{A}k\]and:\[E_{p}:=k(\tau^{s})\subset D_{p}.\]
\subsection{Counting supersingular primes}
We state now a Theorem which provides a criterion to describe the supersingular primes of $\Phi$. See \cite{Goss}, Proposition 4.12.17 for the complete statement.
\begin{thm}
Let $\mathbb{D}_{p}=(\mathbb{G}_{a},\Phi^{v_{p}})$ be the rank $d$ Drinfeld module obtained by reducing modulo $p(T)$ the characteristic 0 one $\mathbb{D}=(\mathbb{G}_{a},\Phi)$ defined over $k$. We then have the following equivalences:
\begin{enumerate}
	\item $p(T)$ is a supersingular reduction prime of $\mathbb{D}$.
	\item 
	There is only one place in $E_{p}$ dividing $p(T)$.
\end{enumerate}
\end{thm}
\begin{proof}
See \cite{Goss}, Proposition 4.12.17.
\end{proof}
We now apply to this description the \textbf{Chebotarev Effective Density Theorem} for function fields (see \cite{Fried-Jarden}, Proposition 6.4.8). We send the reader to such a reference for all details.\\\\
Given $L$ a finite and Galois extension of $k$ we call $G(L/k)$ its Galois group, and for every $p(T)\in S(A)$ which is unramified in $L$ we call $\left(\frac{L/k}{p}\right)$ the corresponding \textbf{Artin symbol}. We recall (see the Introduction) that $P_{N}(A)$ is the set of all monic and irreducible $p(T)\in A$ such that $\deg_{T}(p(T))=N$, for any given $N\in \mathbb{N}\setminus\{0\}$. For any such an $N$, given a conjugacy class $\mathscr{C}$ in $G(L/k)$ we define:\[C_{N}(L/k, \mathscr{C}):=\{p(T)\in P_{N}(A),\textsl{ }p(T)\textsl{ unramified in }L/k,\textsl{ }\left(\frac{L/k}{p}\right)=\mathscr{C}\}.\]
\begin{thm}
Let $L$ be a finite Galois extension of $k$. Let $\mathbb{F}_{q^{\eta}}$ be the algebraic closure of $\mathbb{F}_{q}$ in $L$, and let $\mu:=[L:k\mathbb{F}_{q^{\eta}}]$. Let $\mathscr{C}$ be a conjugacy class in the Galois group $G(L/k)$. Let $a\in \mathbb{N}$ 
be such that:\[\sigma|_{\mathbb{F}_{q^{\eta}}}=\tau^{a}|_{\mathbb{F}_{q^{\eta}}}\]for each $\sigma\in \mathscr{C}$. 
\begin{enumerate}
	\item If $N\not\equiv a\texttt{  } (\eta)$, then $C_{N}(L/k, \mathscr{C})=\emptyset$.
	\item If $N\equiv a\texttt{  } (\eta)$, then $|C_{N}(L/k, \mathscr{C})|\thicksim_{N\to +\infty} \frac{|\mathscr{C}|q^{N}}{N\mu}$.
\end{enumerate}
\end{thm}
\begin{thm}
If a Drinfeld module $\mathbb{D}=(\mathbb{G}_{a},\Phi$) with characteristic $0$ and coefficients in $k$ has rank $d=1$ or a prime number different from the field characteristic, and if it is CM, then it is either 
RV($1,1/2d$)$^{*}$ or RV$_{d}$($1,1/2$)$^{*}$. 
\end{thm}
\begin{proof}
Let us call $L:=End_{k}(\Phi)\otimes_{A}k$. It is not hard to prove that $L$ is actually a field (see for example \cite{D}, Proposition 1.4.14). Moreover, $L/k$ is normal: given $\sigma\in Aut(\ov{k}/k)$ and $P(\tau)\in End_{k}(\Phi)$, for each $a\in A$ we have $\sigma(\Phi_{a}P(\tau))=\sigma(P(\tau)\Phi_{a})=\sigma(P(\tau))\Phi_{a}=\Phi_{a}\sigma(P(\tau))$. Lastly, as $[L:k]=d$ is a prime number different from the field characteristic, it follows immediately that the field extension $L/k$ has to be separable. Hence $L/k$ is a Galois field extension and Theorem 5 applies to it.\\
We know that every maximal field in $D_{p}$, for each $p(T)\in$ $S(A)$, always contains $E_{p}$ (by \cite{Goss}, Theorem 4.12.7, $D_{p}$ is central over $E_{p}$), and that it has degree at most $d$ over $k$. As $\mathbb{D}$ is CM, the field extension $L/k$ has degree $d$ and coincides, embedded in $D_{p}$ via $\widetilde{\Phi^{v_{p}}}$, with a maximal field of $D_{p}$. 
For each $p(T)\in$ $S(A)$, $L$ will thus always contain $E_{p}$. 
This implies that the primes of $A$ which remain inert in $L$ also remain inert in $E_{p}$. 
In order to obtain a lower bound of the supersingular primes of $\Phi$ it is therefore sufficient, by Theorem 4, to focus only on the primes of $A$ which remain inert in $L$. 
Now, all the primes of $A$ which are inert in $L$ are exactly those whose decomposition group (cyclic because they do not ramify in $L$) is $G(L/k)=\langle \sigma\rangle$, by calling $\sigma$ the generator of such a group. Note that our hypothesis that $d$ is a prime number is crucial. In case the cardinality $d$ of $G(L/k)$ were not a prime number, this group might not be cyclic, which would make empty the set of primes of $A$ which remain inert in $L$. Hence, the primes we are looking for are precisely all $p(T)\in$ $S(A)$ such that $\left(\frac{L/k}{p}\right)=\{\sigma\}$. Following the same notations as in Theorem 5 we have:\[\xymatrix{k=\mathbb{F}_{q}(T)\ar@{-}[r]^{\eta}& \mathbb{F}_{q^{\eta}}(T)\ar@{-}[r]^{\mu}&L\simeq \mathbb{F}_{q^{\eta}}(T^{1/\mu})}\]with $L/k$ finite and Galois cyclic extension of degree $d=\eta\mu$, which implies that $G(L/\mathbb{F}_{q^{\eta}}(T))\simeq \mathbb{Z}/\mu\mathbb{Z}$. Clearly, as we are assuming that $d$ is a prime number, it follows that either $\eta=1$ either $\eta=d$. As such extensions are cyclic, the restriction of $\sigma$ to $\mathbb{F}_{q^{\eta}}$ is the generator of the cyclic subgroup $G(L/\mathbb{F}_{q^{\eta}}(T))\simeq \mathbb{Z}/\mu \mathbb{Z}$. The integer $a$ such that res$_{\mathbb{F}_{q^{\eta}}}\tau^{a}=$res$_{\mathbb{F}_{q^{\eta}}}\sigma$ is therefore always $1$. 
If $\eta=1$:\[L\simeq \mathbb{F}_{q}(T^{1/d})\Longrightarrow\texttt{   }|C_{N}(L/k, \{\sigma\})|\thicksim_{N\to +\infty} \frac{q^{N}}{dN}.\]Hence $\mathbb{D}$ is RV($1,1/2d$)$^{*}$ in this case. If on the other hand $\eta=d$ (which means that $\mu=1$ and $End_{k}(\Phi)=\mathbb{F}_{q^{d}}[T]$), we will have that:\[N\not\equiv 1\texttt{   }(d)\Longrightarrow C_{N}(L/k, \{\sigma\})=\emptyset\]\[N\equiv 1\texttt{   }(d)\Longrightarrow |C_{N}(L/k, \{\sigma\})|\thicksim_{N\to +\infty} \frac{q^{N}}{N}.\]This shows that $\mathbb{D}$ is RV$_{d}$($1,1/2$)$^{*}$ in this case. 

\end{proof}
\textsl{Remark}: Let $\mathbb{D}$ be a Drinfeld module satisfying the same hypotheses as in Theorem 6. We have seen in the proof that for $N$ sufficiently large with $N\equiv 1$ mod $(\eta)$, the number of supersingular primes $p(T)\in S(A)$ of degree $N$ is at least $\frac{q^{N}}{2N\mu}$. It easily follows from this that, for $N$ sufficiently large (without any congruence condition), the number of supersingular primes of degree $\leq N$ is at least $c_{\mathbb{D}}\frac{q^{N}}{N}$, for a certain real number $c_{\mathbb{D}}>0$ depending on $\mathbb{D}$. This gives an estimate in the same spirit as in C. David's work \cite{C. David}.

\[\]\[\]\[\]\[\]2010 Mathematics Subject Classification Codes: 11 G 09 and 11 G 50
\end{document}